\documentclass{article}
\usepackage{amsfonts}
\usepackage{mathrsfs}
\usepackage{amssymb, latexsym}
\usepackage{graphicx}
\usepackage{amsmath}
\usepackage[pdftex]{color} 
\definecolor{flesh}{rgb}{.90,.60,.50}
\definecolor{night}{rgb}{.05,0,.05}
\definecolor{gold}{rgb}{.80,.75,.10}

\definecolor{wine}{rgb}{1.00,0,.10}
\definecolor{dimmed}{gray}{0.9}

\definecolor{blau1}{rgb}{0,0.3,0.65}
\definecolor{red1}{rgb}{.5,.0,.0}
\definecolor{green1}{rgb}{.0,.5,.0}
\definecolor{water}{rgb}{.0,.9,0.7}
\definecolor{panelbackground}{rgb}{0,0.39,0.61}
\definecolor{blue1}{rgb}{0.367,0.679,0.839}
\definecolor{blue2}{rgb}{0.000,0.391,0.605}
\definecolor{blue3}{rgb}{0.4,0.8,0.9}

\def\en t{{{\rm Z}\mkern-5.5mu{\rm Z}}}

\newtheorem{theorem}{Theorem}[section]

\newtheorem{definition}[theorem]{Definition}

\newtheorem{example}[theorem]{Example}

\newtheorem{lemma}[theorem]{Lemma}

\newtheorem{proposition}[theorem]{Proposition}

\newtheorem{remark}[theorem]{Remark}

\newenvironment{proof}[1][Proof]{\textbf{#1.} }{\ \rule{0.5em}{0.5em}}

\usepackage{CJK}  

\begin{document}

\begin{CJK*}{UTF8}{gbsn} 

\title{\large\bf On the Navier-Stokes equations and  the Hamilton-Jacobi-Bellman equation 
on the group of volume preserving diffeomorphisms }
\author{Xiang-Dong Li\thanks{Research supported by National Key R$\&$D Program of China (No.~2020YF0712702), NSFC No. 12171458, and Key Laboratory RCSDS, CAS, No. 2008DP173182.}, \ \ Guoping Liu\thanks{Research supported by NSFC No. 11801196 and the Fundamental Research Funds for Central Universities, No. 2018KFYYXJJ043.}}

\maketitle


\numberwithin{equation}{section}

\begin{abstract}
In this paper, we give a new derivation of the incompressible 
Navier-Stokes equations on  a compact Riemannian manifold $M$ via the Bellman dynamic programming principle on the infinite dimensional group $SG={\rm SDiff}(M)$ of 
volume preserving diffeomorphisms.  In particular, when the viscosity vanishes, we give a new derivation of the incompressible Euler equation on  a compact Riemannian manifold. The main result of this paper indicates an interesting relationship among the incompressible Navier-Stokes equations on $M$, the Hamilton-Jacobi-Bellman equation and the viscous Burgers equation  on $SG={\rm SDiff}(M)$. 
This extends Arnold's famous theorem on the geometric interpretation of the incompressible Euler equation on a compact Riemannian manifold $M$ by the geodesic equation on the group $SG={\rm SDiff}(M)$ of volume preserving diffeomorphisms.

\end{abstract}

\medskip


\section{Introduction}\label{Section1}

Hydrodynamics is one of the fundamental areas in mathematics and physics. In  \cite{Arnold66} , Arnold gave a geometric interpretation of the incompressible Euler equation on a Riemannian manifold $M$ as a geodesic on the group $SG={\rm SDiff}(M)$ of volume preserving diffeomorphisms, a subgroup of $G={\rm Diff}(M)$ of diffeomorphisms on $M$, equipped with the $L^2$-right 
invariant Riemannian metric. See also \cite{AK}.

In \cite{EM70} (see also Taylor \cite{Taylor}), the local 
existence and uniqueness of the Cauchy problem to the incompressible Euler equation (i.e., Arnold's geodesic)  within 
suitable Sobolev space was well-established. 
When ${\rm{dim}\hspace{0.3mm}} M=2$, the global existence and uniqueness is known. See e.g. \cite{Taylor}. When ${\rm{dim}\hspace{0.3mm}} M\geq 3$, it remains a major outstanding problem whether the Euler equation has global smooth solutions. On the other hand, concerning the minimal geodesic 
linking two points in $SG$,  Ebin and Marsden  \cite{EM70}  proved that if $h$ belongs to a sufficiently small 
neighborhood of the identity map for a suitable Sobolev norm, 
there is a unique minimal geodesic connecting $h$ and the identity map. However,  a striking result of 
Shnirelman \cite{Shnirelman} shows that, for $M=[0, 1]^3$ and for a large class of data, 
there is no minimal geodesic. This shows
the difficulty of applying the variational method to construct nonstationary  incompressible flows in the
three-dimensional case. In recent years, there have been some new investigations on the incompressible Euler equation. For instance, Brenier \cite{Brenier:99} introduced the notion of  generalized solutions to the imcompressible Euler equation and proved the existence of the generalized solutions. 
He also proved a regularity estimate on the pressure of the generalized solution to the Euler equation. In \cite{Ambrosio-Figalli:10}, Ambrosio and Figalli improved Brenier's regularity estimates for the gradient of the pressure. 

When the fluid is viscid,  the incompressible Navier-Stokes equations on $\mathbb{R}^d$ take the form
	\begin{eqnarray}
	      \begin{cases}
	      $$\partial_t u+\left(u\cdot \nabla \right)u-\nu\Delta u=-\nabla p,$$&\mbox{$$}\\
	      $$\rm{div}~ \hspace{0.3mm} u=0,$$&\mbox{$$}
	      \end{cases}
	      \label{NS1}
	\end{eqnarray}
	where $\nu>0$ is the viscosity constant. In this situation, the system is no longer conservative because of the friction term.  Local existence with initial data $u_0\in L^{n}(\mathbb{R}^d)$ was established in Kato \cite{Kato84}.  Ebin and Marsden \cite{EM70} proved the local existence and uniqueness of smooth solutions to the incompressible Navier-Stokes equations on Riemannian manifolds. See also Taylor \cite{Taylor}. 
When ${\rm{dim}\hspace{0.3mm}} M=2$,
global existence of smooth solutions is also known.  When ${\rm{dim}\hspace{0.3mm}}M \geq 3$, 
 it is a long time outstanding open problem (one of the Millennium Prize 
Problems) whether the incompressible Navier-Stokes equations have global smooth solutions. 

It is interesting to ask the question whether one can extend Arnold's point of view to give a new derivation of the incompressible Navier-Stokes equations $(\ref{NS1})$.
It is well-known that the Laplacian $\Delta$ is the infinitesimal generator of the Brownian motion. So, it is reasonable to derive the incompressible Navier-Stokes equations  by adding the effect of the Brownian motion to the 
Euler equation. In some sense, this 
gives an explanation of the physical term ``internal friction'' (Landau-Lifshitz \cite{Lan}, see
also Serrin \cite{Serrin}) by a probabilistic point of view. To our knowledge, the connections
between Navier-Stokes equations and stochastic evolution traced back to Chorin \cite{Cho}, and a stochastic Hamiltonian approach was given by  Inoue and Funaki \cite{Inoue-Funaki:79}. Since then,  many people have studied the incompressible Navier-Stokes equation using a stochastic Lagrangian variational calculus together with Nelson's stochastic mechanics  \cite{Nelson:79}. For this, we mention the works by Yasue \cite{Yasue:81, Yasue:83}, Nakagomi-Yasue-Zambrini\cite{Nakagomi-Yasue-Zambrini:81} and  Esposito et al \cite{EMPS}. 

In recent years, more probabilistic works have been developed for the incompressible Navier-Stokes equations. In  \cite{LeJan-S}, Le Jan and
Sznitman used a backward-in-time branching process in Fourier space to express the
velocity field of a three-dimensional viscous fluid as the average of a stochastic process,
which then leads to a new existence theorem. In \cite{Bus}, Busnello introduced a  probabilistic approach to the existence of a unique global solution for two dimensional Navier-Stokes equations. In \cite{Bus-Fland}, Busnello, Flandoli and Remito gave the  probabilistic representation formulas for the vorticity and the velocity of
three dimensional Navier-Stokes equations and gave a new proof of the local existence of  solution. In \cite{Cipriano-Cruzeiro:07}, Cipriano and Cruzeiro gave a stochastic variational principle for two dimensional
incompressible Navier-Stokes equations by using the Brownian motions on the
group of homeomorphisms on the torus. In \cite{CS21}, Cruzeiro and Shamarova 
established a connection between the strong solution to the spatially periodic Navier-Stokes equations and a solution to a system of forward-backward stochastic differential equations on the group of volume preserving diffeomorphisms of a flat torus. They also constructed representations of the strong solution to the Navier-Stokes equations in terms of diffusion processes. In \cite{Constantin-Iyer:08}, Constantin and Iyer derived a probabilistic representation of the three-dimensional Navier-Stokes equations based on stochastic Lagrangian paths. As they proved, the particle trajectories obey SDEs driven by a uniform Wiener process, the inviscid Weber formula for the Euler equation of ideal fluids is used to recover the velocity field. This method admits a self-contained proof of local existence for the nonlinear stochastic system and can be extended to formulate stochastic representations of related hydrodynamic-type equations, including viscous Burgers equations and Lagrangian-averaged Navier-Stokes alpha models. By reversing the time variable, Zhang \cite{Zhang:10} derived a stochastic representation for backward incompressible Navier-Stokes equations in terms of stochastic Lagrangian paths and gave a self-contained proof of local existence of solutions in Sobolev spaces in the whole space. He also gave an alternative proof to the global existence for the two-dimensional incompressible Navier-Stokes equations with 
large viscosity. For extension of the above results to compact Riemannian manifolds, see  Fang-Zhang\cite{Fang-Zhang:06}, Arnaudon-Cruzeiro\cite{Arnaudon-Cruzeiro:12}, Arnaudon-Cruzeiro-Fang \cite{ACF}, Fang-Luo\cite{Fang-Luo:15}, Luo\cite{Luo:15} and Fang \cite{Fang}. Owing to the limit of the paper, we will not try to give a full description and citation of all previous works on the probabilistic studies for incompressible Navier-Stokes equations.

The purpose of this paper is to use the Bellman dynamic programming principle on the infinite dimensional group of volume 
preserving diffeomorphisms on a compact Riemannian manifold to derive  the
incompressible Navier-Stokes equations on a compact Riemannian manifold. We would like to point out that, although there is already an extensive literature on the Euler-Lagrangian formulation of the
incompressible Navier-Stokes equations, our work has the following three novelties (see Theorem \ref{Main} and Remark \ref{forwardNS} below):

1. We prove that there exists a one to one correspondence  between the set of smooth
solutions to the   incompressible Navier-Stokes equations on a compact Riemannian manifold $M$ and the set of 
solutions to the viscous Burgers equation on $SG={\rm SDiff}(M)$.
More precisely,  $u(t, x)$ is a smooth solution to the incompressible Navier-Stokes equations $(\ref{NSM-u1})$ and  $(\ref{NSM-u2})$ on a compact Riemannian manifold $M$, if and only if $U(t, g)=(dR_g) U(t, e)$ is a smooth solution to the viscous Burgers equation $(\ref{BurgersG}$) on $SG={\rm SDiff}(M)$, where $dR_g$ is the differential of the right translation  $R_g: SG\rightarrow SG$, and $U(t, e)=\{u(t, x): x\in M\}$. This is indeed the extension of Arnold's famous theorem on the geometric interpretation of the incompressible Euler equation on a compact Riemannian manifold $M$ by the geodesic equation on the group $SG={\rm SDiff}(M)$ of volume preserving diffeomorphisms. 

2. The value function $W$ for an optimal stochastic control problem on $SG={\rm SDiff}(M)$  satisfies the Hamilton-Jacobi-Bellman equation (see $(\ref{HJB-W})$ below) on $SG={\rm SDiff}(M)$, and the optimal Markov control $U=-\nabla^{SG} W$ on $SG={\rm SDiff}(M)$  satisfies the backward viscous Burgers equation $(\ref{BurgersG}$) on $SG={\rm SDiff}(M)$.

3. As a consequence,  we prove that $u(t, x)=-(de_{x}) \nabla^{SG} W(t, e)$ gives us an explicit solution to the backward incompressible Navier-Stokes equation $(\ref{NSM-u1})$ and  $(\ref{NSM-u2})$ on $M$, where $de_x: T_e {SG}\rightarrow T_xM $ is the differential of the evaluation mapping $e_x: SG\rightarrow M, g\mapsto g(x)$ at $g=e$ (the unit element in $SG$).  By time reversal, $-u(T-t, x)=(de_{x}) \nabla^{SG} W(T-t, e)$ gives us an explicit solution to the forward incompressible Navier-Stokes equations $(\ref{NSMF1})$ and  $(\ref{NSMF2})$ on $M$.

To the best knowledge of the authors, this is a new point of view even in the case where the viscosity coefficient vanishes. That is to say, the value function $W$  for a deterministic optimal control problem on $SG={\rm SDiff}(M)$ satisfies the Hamilton-Jacobi equation (see $(\ref{HJ-W})$ below) on $SG={\rm SDiff}(M)$, and $u(t, x)=-(de_{x}) \nabla^{SG} W(t, e)$ gives us an explicit solution to the incompressible Euler equation on $M$. See Theorem \ref{Main-Euler} below.  These suggest a deep and close connection between the incompressible Navier-Stokes (respectively,  Euler) equations and the Hamilton-Jacobi-Bellman (respectively, the Hamilton-Jacobi) equation via the Bellman dynamic programming  principle on the group of volume preserving diffeomorphisms. We hope that our work may bring  some new point of views for the further study of the incompressible Navier-Stokes and Euler equations.  

The rest of the paper is organized as follows. In Section 2, we extend the 
Bellman dynamic programming principle and derive the Hamilton-Jacobi-Bellman  equation on compact Riemannian manifolds. In Section 3, we  extend the  
Bellman dynamic programming principle to the infinite dimensional group
$SG={\rm SDiff}(M)$ of volume preserving diffeomorphisms. In Section 4, we derive the incompressible  Navier-Stokes equations 
on a compact Riemannian manifold using the 
Bellman dynamic programming principle and the Hamilton-Jacobi-Bellman   equation on $SG={\rm SDiff}(M)$. In Section 5, we give a new derivation of the incompressible Euler equation via the deterministic dynamic programming principle and the Hamilton-Jacobi equation on $SG={\rm SDiff}(M)$. In Section 6, we formulate our main result which indicates an interesting relationship among the incompressible Navier-Stokes equations on $M$, the Hamilton-Jacobi-Bellman 
equation and viscous Burgers equation  on the group $SG={\rm SDiff}(M)$  of volume preserving diffeomorphisms, and we raise some problems  for further research work.

\section{Dynamic programming principle and HJB equation on Riemannian manifolds}

In this section, we extend the classical  Bellman dynamic programming principle (briefly, DPP) from the Euclidean space to compact Riemannian manifolds, and derive the Hamilton-Jacobi-Bellman (briefly, HJB) and viscous Burgers equations on Riemannian manifolds.

First, for the convenience of the reader, we briefly review the standard results on the Bellman dynamic programming principle and the Hamilton-Jacobi-Bellman equation in stochastic control theory. 

Consider the following controlled SDE on $\mathbb{R}^d$
	\begin{eqnarray}
	      \begin{cases}
	      $$dx^u(t)=\sqrt{2\nu}dB_t+u(t, x^u(t))dt,$$&\mbox{$$}\\
	      $$x^u(0)=x,$$&\mbox{$$}
	      \end{cases}
	      \label{SDER}
	\end{eqnarray}
where $B_t$ is a $\mathbb{R}^d$-valued standard Brownian motion. The
corresponding action function on $\mathbb{R}^d$ is given by
\begin{eqnarray}
     S[x^u]=\mathbb{E}\left[\int_0^T \frac{1}{2}|D_t x^u(t)|^{2}_{\mathbb{R}^d} dt\right].
\end{eqnarray}
Following \cite{Nelson:79},  the Nelson derivative of the diffusion process $x^u(t)$ is defined as follows
\begin{eqnarray}\label{NelR}
D_tx^u(t):=\lim\limits_{\varepsilon\rightarrow 0} \mathbb{E}\left[\left. {dx^u(t+\varepsilon)-dx^u(t)\over \varepsilon} \right|\mathcal{F}_t\right].
\end{eqnarray}

Let $\mathcal{U}$ 
be the set of all time dependent adapted smooth vector fields $X$ on $\mathbb{R}^d$ which satisfies the $L^2$-integrability condition 
\begin{eqnarray}
      \|X\|_{L^2}^2=\mathbb{E}\left[\int_0^T |X(s, x(s))|_{\mathbb{R}^d}^2  ds\right]<+\infty.
\end{eqnarray}

Let  $\Psi\in C(\mathbb{R}^d, \mathbb{R})$, and define  the value function
\begin{eqnarray}
     W(t, x)=\inf_{u\in \mathcal{U}}\mathbb{E}\left[\left.\int_t^T \frac{1}{2}|D_s x^u(s)|_{\mathbb{R}^d}^2 ds+\Psi(x^u(T))\right|x^u(t)=x\right].
\end{eqnarray}
By Bellman's principle of stochastic dynamic programming on Euclidean space, it is well-known that $W$ satisfies the Hamilton-Jacobi-Bellman equation
\begin{eqnarray}
     \partial_t W+\nu\Delta W-\frac{1}{2}|\nabla W|^2=0, \label{HJB-1}
\end{eqnarray}
and the optimal Markov control is given by 
\begin{eqnarray}
      u(t, x)=-\nabla W(t, x).
\end{eqnarray}
See e.g. Zambrini \cite{Zambrini86}, 
Fleming and Soner \cite{Fleming-Soner}. 
Suppose that $W\in C^{1, 3}([0, T]\times \mathbb{R}^d, \mathbb{R})$.  Differentiating the both sides of $(\ref{HJB-1})$ with respect to spatial variable $x\in \mathbb{R}^d$, we derive the viscous Burgers equation on $\mathbb{R}^d$
\begin{eqnarray}
     \partial_t u+\nu\Delta u+\nabla_uu=0. \label{Euclidean-Burges}
\end{eqnarray}

\subsection{Dynamic programming principle on Riemannian manifolds}

Let $(M, g)$ be a $d$-dimensional compact Riemannian manifold without boundary, $g$ the Riemannian metric, and $dv=\sqrt{{\rm det}g(x)}dx$ the volume measure on $(M, g)$. Let $\nabla^{TM}$ be the Levi-Civita covariant derivative operator on $TM$, and $\nabla^M$ (respectively, $\Delta^M$) the Riemannian gradient (respectively, the Laplace-Beltrami) operator on $(M, g)$.
Let $(\Omega, \mathcal{F},(\mathcal{F}_s),\mathbb{P})$ be a complete probability space endowed with an increasing
filtration  $(\mathcal{F}_s)$ which satisfies the usual condition.

Let $u\in C^{1, 2}([0, T]\times M, TM)$. Let $x_0\in M$ be fixed, and $B_t$ be a standard Brownian motion on  $T_{x_0}M$. 
By It\^o's SDE theory,  the following controlled stochastic differential equation on $M$

\begin{eqnarray}
      \begin{cases}
      $$dx(s)=\sqrt{2\nu}U_{t\rightarrow s} \circ dB_{t}+u(s, x(s))ds,$$&\mbox{$t\leq s\leq T$}\\
      $$x(t)=x,$$&\mbox{$x\in M$}
      \end{cases}
       \label{SDEM}
\end{eqnarray}
admits a  unique strong solution, where $U_{t\rightarrow s}: T_{x(t)}M\rightarrow T_{x(s)}M$ denotes the stochastic parallel transport along the trajectory of the diffusion process $\{x(r): r\in [t, s]\}$, which satisfies the covariant SDE on $M$
\begin{eqnarray}\label{ch2-eq:U}
    \nabla_{\circ dx(s)} U_{t\rightarrow s}=0, \ \ \ U_{t\rightarrow t}={\rm id}_{T_{x(t)} M}, \ \ \forall \ t\leq s\leq T,
     \end{eqnarray}
    where $\circ$ denotes the Stratonovich differentiation.  See \cite{Ma97, ELL}.

Let $\mathcal{U}^M$ be the set of all time dependent adapted smooth vector fields $X$ on $M$ which satisfies the $L^2$-integrability condition 
\begin{eqnarray}
      \|X\|_{L^2}^2=\mathbb{E}\left[\int_0^T |X(s, x(s))|_{T_{x(s)}M}^2ds\right]<+\infty.
\end{eqnarray}

	Given a Lagrangian function $L\in C^1([0, T]\times TM, \mathbb{R})$ and $\Psi\in C(M, \mathbb{R})$, 
we consider the following minimization problem on a finite time interval $t\leq s\leq T$ 

\begin{eqnarray}\label{ch2-eq:2}
       {\rm minimizes}\ J(t,x;u)=\mathbb{E}_{t,x}\left[\int_t^TL(s,x(s),u(s))ds+\Psi(x(T))\right],
\end{eqnarray}
subject to $u\in \mathcal{U}^M=\left\{u  {\rm{~is~ (\mathcal{F}_s)~adapted~}}:\mathbb{E}\left[\int_0^T |u(s, x(s))|_{T_{x(s)}
 M}^2ds\right]<\infty \right\}$.  For simplicity, we write $\mathbb{E}_{t, x}[\cdots]$ instead of 
$\mathbb{E}\left[\left. \cdots\right|x(t)=x\right]$ throughout this paper.

By the same argument as in the proof of the dynamic programming equation on the Euclidean spaces,  
see e.g. Fleming and Soner\cite{Fleming-Soner}, we can prove the following

\begin{theorem} \label{Bellman} Let $M$ be a compact Riemannian manifold without boundary. Let $W$ be  the value function defined by 

\begin{eqnarray}\label{DPPM}
      W(t, x)=\inf\limits_{u\in\mathcal{U}^M}J(t, x; u).
\end{eqnarray}
Assume that $L\in C^{1}([0, T]\times TM, \mathbb{R})$ and $W\in  C^{1, 2}([0, T]\times M,\mathbb{R})$. 
Then the dynamic programming equation holds

\begin{eqnarray}\label{ch2-eq:6}
0=\inf\limits_{v\in\mathcal{U}^M} \left[A^v W(t, x)+L(t, x, v) \right],
\end{eqnarray}
where
	\begin{eqnarray*}
A^v=\partial_tW+\nu\Delta^M W+\langle v, \nabla^M W\rangle_{T_x M}.
\end{eqnarray*}

	\end{theorem}

\subsection{The Hamilton-Jacobi-Bellman  equation on Riemannian manifolds}

Following \cite{Nelson:79}, see also \cite{Cipriano-Cruzeiro:07, Arnaudon-Cruzeiro:12, ACF}, the Nelson derivative of the diffusion process $x(t)$ on $M$ is defined as follows
\begin{eqnarray}\label{Nel}
D_tx(t):=\lim\limits_{\varepsilon\rightarrow 0} \mathbb{E}\left[\left. {U_{t+\varepsilon\rightarrow t} \circ dx(t+\varepsilon)-dx(t)\over \varepsilon} \right|\mathcal{F}_t\right].
\end{eqnarray}
Note that, by \cite{Nelson:79, ACF},  for $x(t)$ defined by SDE $(\ref{SDEM})$, the Nelson derivative of $x(t)$ is given by 

$$
D_t x(t)=u(t, x(t)).$$

We now state the main result of this subsection.

\begin{theorem} \label{BurgersM} Given $V\in C^1(M, \mathbb{R})$ and $\Psi\in C(M, \mathbb{R})$. Let \begin{eqnarray}
        L(s, x(s),u(s))=\frac{1}{2}|u(s, x(s))|_{T_{x(s)}M}^2-V(x(s))
\end{eqnarray} 
be the Lagrangian function, and let
\begin{eqnarray*}
     W(t, x)=\inf_{u\in \mathcal{U}^M}\mathbb{E}_{t, x}\left[\int_t^T L(s, x(s), u(s))ds+\Psi(x(T))\right]
\end{eqnarray*}
be the value function. Suppose that  $W \in C^{1, 3}([0, T]\times M, \mathbb{R})$. Then $W$ satisfies the Hamilton-Jacobi-Bellman equation on $M$
\begin{eqnarray}
\partial_tW+\nu  \Delta^M W-\frac{1}{2}|\nabla^M W|^2_{T_{x} M}-V(x)=0, \label{ch2-eq:9}
\end{eqnarray}
with the terminal condition $W(T, x)=\Psi(x)$.
Moreover, the optimal  control is given by 
\begin{eqnarray}
u^*(t, x)=-\nabla^M W(t,  x)\label{vW}.
\end{eqnarray}

Define \begin{eqnarray}
v(t, x)=\nabla^M W(T-t, x).   \label{ch2-eq:9v}
\end{eqnarray}
Then $v$ satisfies the viscous Burgers equation on $M$
\begin{eqnarray}
     \partial_t v+\nu\square^M v+\nabla_{v}^{TM}v+\nabla^M V=0, \label{M-Burgers}
\end{eqnarray}
where $\square^M=-\Delta^M+{\rm Ric^M}$ is the Hodge Laplacian, $\Delta^M={\rm Tr}(\nabla^M)^2$ is the covariant Laplacian, ${\rm Ric}^M$ is the Ricci curvature on $M$, and $\nabla^{TM}$ denotes the Levi-Civita covariant derivative operator on $TM$. 
\end{theorem}

\begin{proof} The Bellman stochastic dynamic programming principle on $M$ says that the value function $W$ satisfies the corresponding dynamic programming equation 
\begin{eqnarray}\label{ch2-eq:9"}
       0=\inf_{v\in \mathcal{U}^M}\left\{\partial_tW+\nu\Delta^M W+\langle v, \nabla^M W\rangle_{T_x M}+\frac{1}{2}|v|_{T_x M}^2-V(x)\right\}.
\end{eqnarray}

Using the Riemannian inner product on the tangent space,  we have
\begin{eqnarray*}
& &\inf_{v\in \mathcal{U}^M}\left\{\partial_tW+\nu\Delta^M W+\langle v, \nabla^M W\rangle_{T_x M}+\frac{1}{2}|v|^2_{T_{x} M}-V(x)\right\}\\
	&=&\partial_tW+\nu\Delta^M W+\inf_{v\in \mathcal{U}^M}\left\{\langle v, \nabla^M W\rangle_{T_x M}+\frac{1}{2}|v|^2_{T_{x} M}\right\}-V(x)\\\nonumber
&=&\partial_tW+\nu\Delta^M W-\frac{1}{2}|\nabla^M W|^2_{T_{x} M}-V(x). 
\end{eqnarray*}
This proves that $W$ satisfies the Hamilton-Jacobi-Bellman equation $(\ref{ch2-eq:9})$ on $M$. 
Moreover, from the above argument, we see that the optimal Markov control $u^*$ is given by
\begin{eqnarray*}
u^*=-\nabla^M W.
\end{eqnarray*}

Assuming that $W\in C^{1, 3}([0, T]\times M)$, and taking the covariant differentiation 
on both sides of  the Hamilton-Jacobi-Bellman equation  $(\ref{ch2-eq:9})$, we have

\begin{eqnarray}
\partial_t \nabla^M W+\nu\nabla^M\Delta^M W-\frac{1}{2}\nabla^M |\nabla^M W|_{T_xM}^2-\nabla^M V(x)=0. \label{ch2-eq:9b}
\end{eqnarray}

Using the commutation formula in differential geometry
\begin{eqnarray}
\Delta^M \nabla^M  W=\nabla^M \Delta^M W+{\rm Ric}^M(\nabla^M W), \label{commutative-1}
\end{eqnarray} 
and the geometric formula
\begin{eqnarray}
{1\over 2}\nabla^M |\nabla^M W|_{T_xM}^2=\langle\nabla^M\nabla^M W, \nabla^M W\rangle=\nabla^{TM}_v v, \label{commutative-2}
\end{eqnarray}
we derive that $v(t, x)=\nabla^M W(T-t, x)$ satisfies the viscous Burgers equation $(\ref{M-Burgers})$. 
\end{proof}

\subsection{Regularity of solution to HJB equation on $M$}

Note that $(\ref{M-Burgers})$ is a system of quasi-linear   parabolic partial differential equations on compact Riemannian manifolds. By standard argument as used in \cite{Taylor} on quasi-linear parabolic PDEs, we can prove the local 
existence and uniqueness of the Cauchy problem to the Burgers equation within 
suitable Sobolev space on a compact Riemannian  manifold $M$.

Following \cite{Fleming-Soner}, under the Cole-Hopf  transformation
\begin{eqnarray}
\Phi(t, x)=\exp{\left[-{W(t, x)\over {2\nu}}\right]}, ~~\varphi(x)=\exp[-\Psi(x)], \label{Hopf}
\end{eqnarray}
the  Hamilton-Jacobi-Bellman equation  $(\ref{ch2-eq:9})$ becomes the heat equation
\begin{eqnarray}
\partial_t\Phi+\nu\Delta^M\Phi+\frac{1}{2\nu}V\Phi=0,\ 
\ \ \Phi(T, x)=\varphi(x), \label{HeatV}
\end{eqnarray}
which is the backward heat equation with a potential term $\frac{1}{2\nu}V\Phi$. 

Conversely, we can use the inverse of the Cole-Hopf  transformation $(\ref{Hopf})$  to derive the Hamilton-Jacobi-Bellman equation  $(\ref{ch2-eq:9})$ from the heat equation $(\ref{HeatV})$ provided that $\Phi$ is regular. Using the Feynman-Kac formula for Eq. $(\ref{HeatV})$,  it is well-known that $\Phi\in C^{1, 2}([0, T]\times M)$ provided that $\varphi\in C^2(M)$. Indeed, we have the following

\begin{theorem} \label{M-continuity}Suppose that the terminal value function $x\mapsto \Psi(x)$ is $C^k$-smooth. Then 
 the value function $x\mapsto W(t, x)$ is also  $C^k$-smooth on $M$.
\end{theorem}
\begin{proof} The proof is standard and is omitted. \end{proof}

\section{Dynamic programming principle on $SG={\rm SDiff}(M)$}

In this section, we give the  Bellman  dynamic programming principle on the group of volume preserving diffeomorphisms. 

Let $(\Omega, \mathcal{F}, (\mathcal{F}_s), \mathbb{P})$ be a filtered probability space satisfying the usual condition. Let $M$ be a $d$-dimensional compact Riemannian manifold, $SG={\rm SDiff}(M)$ the group of volume preserving diffeomorphisms on $M$, and  $\mathcal{SG}=T_e(SG)$ its Lie algebra, i.e.,  
$$\mathcal{SG}=\left\{X\in \Gamma(TM): {\rm div}\hspace{0.2mm}X=0\right\}$$ 
which is the set of divergence free smooth vector fields on $M$. 

Let $Q$ be  a symmetric non-negative Hilbert-Schmidt operator on $\mathcal{SG}$ equipped with the $L^2$-Riemannian metric defined in Section $2$. 

The following definition of $\mathcal{SG}$-valued $Q$-Wiener process is standard. 

\begin{definition} \label{QBM} A continuous $\mathcal{SG}$-valued stochastic process $B^\mathcal{SG}_s$ is called a $Q$-Wiener process or $Q$-Brownian motion if\\

(1) $B^\mathcal{SG}_s=0$ a.s., \\

(2) $B^\mathcal{SG}_s$ has independent increments $B^\mathcal{SG}_{st}=B^\mathcal{SG}_t-B^\mathcal{SG}_s$ for all $0\leq s<t$, \\
 
(3) the increments have the following Gaussian laws: $\mathbb P\circ (B^\mathcal{SG}_{st})^{-1}=
\mathcal N(0, (t-s)Q)$ for all $0\leqslant s< t$.
\end{definition}

Let $\{g(s), s\in [t, T]\}$ be defined as the solution to the following Stratonovich stochastic differential equation on $SG$
\begin{eqnarray} \label{SDEG-1}
      \begin{cases}
      $$dg(s)=b(g(s), U(s))ds+\sigma(g(s), U(s))\circ dB^\mathcal{SG}_s,$$&\mbox{}\\
      $$g(t)=g,$$&\mbox{}
      \end{cases}
\end{eqnarray}
where $b: T(SG)\rightarrow T(SG)$, $\sigma:T(SG)\rightarrow \mathcal {L}(T(SG);T(SG))$ satisfy the Lipschitz condition, $T(SG)$ is the tangent 
bundle on $SG$, $\mathcal {L}(T(SG);T(SG))$ is the bounded linear operator, and $U(s)$ 
is a $T_{g(s)}(SG)$-valued $(\mathcal{F}_s)$-adapted process along the diffusion process $g(s)$ on $SG$. 
We call such  $U$ an admissible control on SDE $(\ref{SDEG-1})$ on $SG$.

Following \cite{Nelson:79}, see also \cite{Cipriano-Cruzeiro:07, Arnaudon-Cruzeiro:12, ACF}, the Nelson derivative of the diffusion process $g(t)$ on $SG$ is defined as follows
\begin{eqnarray*}
D_tg(t):=\lim\limits_{\varepsilon\rightarrow 0} \mathbb{E}\left[\left. {//^{g}_{t+\varepsilon\rightarrow t} \circ dg(t+\varepsilon)-dg(t)\over \varepsilon} \right|\mathcal{F}_t\right]
\end{eqnarray*}
where $//^{g}_{t+\varepsilon\rightarrow t}: T_{g(t+\varepsilon)}(SG)\rightarrow T_{g(t)}(SG)$ denotes the stochastic parallel transport from $T_{g(t+\varepsilon)}(SG)$ to $T_{g(t)}(SG)$. Indeed, it is well-known that 
$$
D_t g(t)=b(g(t), U(t)).$$

Let  
$$\mathcal{U}^{SG}=\{U: [0, T]\times SG \rightarrow T(SG)\ \ {\rm adapted\ to}\ (\mathcal{F}_s):  U(s, g(s))\in T_{g(s)}(SG) \}$$ be the set of admissible controls on $SG$. For simplicity,  we use $\mathbb{E}_{t, g}[\cdots]$ instead of $\mathbb{E}\left[\left.\cdots\right|g(t)=g\right].$

We introduce the value function on  $[t, T]\times SG$ as follows
\begin{eqnarray}\label{ch2-eq:11}
W(t, g)=\inf\limits_{U\in\mathcal{U}^{SG}}\mathbb{E}_{t, g}\left[\int_t^TL(s, g(s), U(s))ds+\Psi(g(T))\right].
\end{eqnarray}

In \cite{FGS}, the authors discussed the
 connection between infinite dimensional stochastic optimal control problems and HJB equations in Hilbert spaces, which extend the well-known results in the finite dimensional Euclidean spaces \cite{Fleming-Rishel, Fleming-Soner, Nisio, Yong-Zhou:99}. 

Following \cite{FGS, Fleming-Rishel, Fleming-Soner, Nisio, Yong-Zhou:99}, we can extend the Bellman dynamic programming principle and the dynamic programming equation to the infinite dimensional group of volume preserving diffeomorphisms.  The proofs can be given following standard argument as 
used in \cite{FGS} and in
\cite{Fleming-Rishel, Fleming-Soner, Nisio, Yong-Zhou:99} for the finite dimensional case. To save the length of the paper, we omit the proofs.

\begin{theorem}\label{DPPG}
Let $\Psi\in C(SG, \mathbb{R})$ and $W(T,g)=\Psi(g)$, $g\in SG$. For $h$ small enough such that $0\leq h\leq T-t$, we have
\begin{eqnarray}\label{ch2-eq:12}
W(t, g)=\inf\limits_{U\in\mathcal{U}^{SG}}\mathbb{E}_{t, g}\left[\int_t^{t+h}L(s,g(s), U(s))ds+W(t+h, g(t+h))\right].
\end{eqnarray}
Furthermore, assume that $W\in C^{1, 2}([0, T]\times SG, \mathbb{R})$, then $W$ is a solution of the Bellman dynamic programming equation

\begin{eqnarray}\label{HJB-4}
0=\inf\limits_{v\in\mathcal{U}^{SG}} \left[A^v W(t, g)+L(t, g, v)\right],
\end{eqnarray}
with the terminal condition
\begin{eqnarray}\label{HJB-6}
W(T, g)=\Psi(g),
\end{eqnarray}
where
\begin{eqnarray}
A^v=\partial_t+b(t,g,v)\cdot \nabla^{SG}+
\frac{1}{2}{\rm Tr}[(\sigma(t,g,v)Q^{\frac{1}{2}})^*(\sigma(t,g,v)Q^{\frac{1}{2}})(\nabla^{SG})^2]
\label{HJB-5}
\end{eqnarray}
and $(b(t,g,v)\cdot \nabla^{SG}) W= \langle\langle b(t, g, v), \nabla^{SG} W(t, g)\rangle \rangle_{T_g(SG)}$ and $\langle\langle \cdot, \cdot\rangle \rangle_{T_g(SG)}$ is the right-invariant $L^2$-Riemannian inner product on $SG={\rm SDiff}(M)$.
\end{theorem}

To simplify the notation, for fixed $Q$, we introduce the $Q$-Laplacian on $SG$ as follows
\begin{eqnarray}
\Delta_Q^{SG}=
\frac{1}{2}{\rm Tr}[(\sigma(t,  g, v)Q^{\frac{1}{2}})^*(\sigma(t,  g,  v)Q^{\frac{1}{2}})(\nabla^{SG})^2]. \label{DeltaQG}
\end{eqnarray}

Let $\mathcal{V}\in C^1(SG, \mathbb{R})$. 
Consider the Lagrangian action function
\begin{eqnarray*}
L(s,  g(s),U(s))=\frac{1}{2}\|U(s)\|_{T_{g(s)}(SG)}^2-\mathcal{V}(g(s)),
\end{eqnarray*}
where $\|\cdot \|_{T_{g(s)}(SG)}$ is the right-invariant $L^2$-Riemannian metric on $SG={\rm SDiff}(M)$.  By the dynamic programming equation $(\ref{HJB-4})$ derived from
Theorem \ref{DPPG},
we have
\begin{eqnarray}\label{HJB-7}
0&=&\inf_{v\in\mathcal{U}^{SG}}\left\{\partial_tW+\nu\Delta_Q^{SG}W+\langle \langle v,\nabla^{SG}W\rangle \rangle_{T_g(SG)}+\frac{1}{2}\|v\|_{\mathcal{SG}}^2-\mathcal{V}(g)\right\} \\ \nonumber
&=&\partial_tW+\nu\Delta_Q^{SG}W-\frac{1}{2}\|\nabla^{SG}W\|_{\mathcal{SG}}^2-\mathcal{V}(g),
\end{eqnarray}
with terminal condition (\ref{HJB-6}). Moreover, the optimal Markov control $U$ is given by
\begin{eqnarray}\label{HJB-8}
U(t, g)=-\nabla^{SG}W(t, g),
\end{eqnarray}
where $\nabla^{SG}$ is the gradient on $SG$, and $\Delta_Q^{SG}$ is the $Q$-Laplacian on $SG$.

In \cite{FGS}, the authors also studied the existence, uniqueness results of viscosity solutions (Theorem 3.66, P231 and  Theorem 3.67,  P237 in \cite{FGS}) and mild solutions (Theorem 4.80, P437, Theorem 4.85, P444 and Theorem 4.90, P445 in \cite{FGS}) of the HJB equation in Hilbert spaces.

\section{New derivation of the incompressible Navier-Stokes equations}

In this section, we first derive the Hamilton-Jacobi-Bellman and the viscous Burgers equations on the infinite dimensional group of volume preserving diffeomorphisms. By pushing forward  to the underlying manifold, and using the Hodge decomposition theorem, we then derive the incompressible Navier-Stokes equations on compact Riemannian manifolds.

\subsection{The idea for the derivation of incompressible Navier-Stokes equations}

In Section $2$, we derive the viscous Burgers equation $(\ref{M-Burgers})$ 
 on compact Riemannian  manifolds from the Bellman dynamic programming principle, which, at least in the case of bounded domain in Euclidean space,  has the similar form as  the incompressible Navier-Stokes equations except the divergence-free constraint condition ${\rm div}\hspace{0.3mm}~u=0$. In a series works of Cruzeiro and her collaborators \cite{Cipriano-Cruzeiro:07, CS21}, they developed a stochastic variational principle for the action 
 \begin{eqnarray}
       \mathcal{A}(x^u, u)&=&\mathbb{E}\left[\int_0^T \frac{1}{2}\|D_t x^u(t)\|_{L^2(M)}^2 dt\right]\nonumber\\
       &=&\mathbb{E}\left[\int_0^T \int_M\frac{1}{2} |D_t x^u(t)(x)|^2 dv dt\right]\label{Action-EL}
\end{eqnarray}
among the diffusion processes $\{x^u(t), t\in [0, T]\}$ defined by  SDE $(\ref{SDEM})$ with control $u\in \mathcal{U}^M$. More precisely, 
Cipriano and Cruzeiro \cite{Cipriano-Cruzeiro:07} proved that the optimal control $u\in C^{1, 2}([0, T]\times M, \Gamma(TM))$ for the Euler-Lagrangian equation of the Lagrangian action $\mathcal{A}$ on the two-dimensional torus $\mathbb{T}^2$ satisfies the incompressible Navier-Stokes equations
\begin{eqnarray*}
\partial_tu+u\cdot \nabla u-\nu \Delta u &=& -\nabla p,\\
{\rm div}\hspace{0.3mm}~u &=& 0.
￼\end{eqnarray*}
See further works in \cite{Arnaudon-Cruzeiro:12, ACF, Fang-Luo:15, Luo:15}  for the extension of this idea to compact Riemannian manifolds. 

 Let $\mathcal{V}: SG\rightarrow \mathbb{R}$ be a $C^1$-continuous function. 
We introduce the Lagrangian action function on $SG={\rm SDiff}(M)$
\begin{eqnarray}
       \mathcal{L}(g^U,U)
       =  \int_M\frac{1}{2} |D_s g^U(s)(x)|^2dv-
       \mathcal{V}(g^U(s)),
       \label{Action-G}
\end{eqnarray}
where $g^U(s)$ is the horizontal lift of the  solution $x^u(s)$ of the Controlled SDE $(\ref{SDEM})$ on  $SG={\rm SDiff}(M)$.

Let $\Psi\in C(SG, \mathbb{R})$. We introduce the  value function $W$ on $SG={\rm SDiff}(M)$
\begin{eqnarray}
     W(t, g)=\inf\limits_{U\in \mathcal{U}^{SG}}\mathbb{E}_{t, g}\left[
     \int_t^T  \left(  \int_M\frac{1}{2} |D_s g^U(s)(x)|^2dv-\mathcal{V}(g^U(s)) \right)ds+\Psi(g^U(T))\right].
\label{WG}
\end{eqnarray}
Here and throughout this paper, $\mathbb{E}_{t, g}[\cdots ]=\mathbb{E}\left[\left. \cdots \right|g^U(t)=g\right]$.

By the Bellman dynamic programming principle on $SG={\rm SDiff}(M)$ in Section $3$, we have
\begin{eqnarray}\label{ch2-eq:9"}
       0&=&\inf_{v\in T_g(SG)}\left\{\partial_tW+\nu\Delta^{SG}W+\langle\langle v, \nabla^{SG}W\rangle\rangle_{T_g(SG)}+\frac{1}{2}\|v\|_{T_g (SG)}^2-\mathcal{V}(g)
       \right\}\\\nonumber
         &=&\partial_tW+\nu\Delta^{SG} W+\inf_{v\in T_g (SG)}\left\{\langle\langle v, \nabla^{SG} W\rangle\rangle_{T_g (SG)}+\frac{1}{2}
         \|v\|_{T_g(SG)}^2\right\} -\mathcal{V}(g),
\end{eqnarray}
where $\nabla^{SG}$ and $\Delta^{SG}$ denote the Riemannian gradient operator and the weighted 
Laplace-Beltrami  (i.e., the Ornstein-Uhlenbeck) type operator on the infinite dimensional Riemannian manifold  $SG={\rm SDiff}(M)$. 

Then, by a similar argument as used in the derivation of the viscous Burgers equation $(\ref{M-Burgers})$ on $M$,  we can prove that the value function $W$ satisfies the  infinite dimensional Hamilton-Jacobi-Bellman equation on $SG={\rm SDiff}(M)$ 
\begin{eqnarray}
\partial_tW+\nu\Delta^{SG}W-\frac{1}{2}\|\nabla^{SG}W\|^2_{T_{g} (SG)}-\mathcal{V}=0, \label{HJBG}
\end{eqnarray}
with the terminal condition $W(T, g)=\Psi(g)$. Moreover, the optimal Markov control $U$ is given by
\begin{eqnarray*}
U=-\nabla^{SG} W.
\end{eqnarray*}
By the same argument as explained in the derivation of the viscous Burgers equation $(\ref{M-Burgers})$, we will prove that $U$ 
satisfies the viscous Burgers  equation on $SG$.

Furthermore, by the Hodge decomposition theorem, and by pushing forward $(\ref{HJBG})$  to $M$ via the evaluation map $e_x: SG\rightarrow M, g\mapsto g(x)$, we can derive that $u=(de_x) U$ satisfies the incompressible Navier-Stokes equations on 
the underlying Riemannian manifold $(M, g)$. 

This gives a new derivation of the incompressible Navier-Stokes equations
via the Bellman dynamic programming principle 
on the infinite dimensional group $SG={\rm SDiff}(M)$. In particular, when the viscosity coefficient $\nu$ vanishes, we give a new derivation of the incompressible Euler equation 
via the deterministic dynamic programming principle 
on the infinite dimensional group $SG={\rm SDiff}(M)$.

\subsection{The Hamilton-Jacobi-Bellman equation on $SG={\rm SDiff}(M)$}

In this subsection, we introduce a controlled SDE on the group  $SG={\rm SDiff}(M)$ of volume preserving diffeomorphisms and derive the Hamilton-Jacobi-Bellman equation on $SG={\rm SDiff}(M)$. 

Let $\{g^U(t), t\in [0, T]\}$ be a diffusion process with values in $SG={\rm SDiff}(M)$ and satisfies the following  Stratonovich 
stochastic differential equation (briefly, SDE) with right action on $SG={\rm SDiff}(M)$
	\begin{eqnarray}
	      \begin{cases}
	      $$dg^U(t)=\left(\sqrt{2\nu}\circ dW_t\right) g^U(t)+U(t, g^U(t))dt,$$ &\mbox{$$}\\
	      $$g^U(0)=e,$$&\mbox{$$}
	      \end{cases}
	      \label{SDEG}
	\end{eqnarray}
where $W_t$ is a $\mathcal{SG}$-valued $Q$-Brownian motion (see Definition \ref{QBM}), $U$ is a horizontal lift of 
a smooth divergence-free vector field $u$ on $M$ to $\mathcal{SG}$ (see Proposition \ref{U-u} below), and $e=id$ is the identity map of $SG={\rm{SDiff}} (M)$.  The short time existence and uniqueness of solution to SDE $(\ref{SDEG})$ follows from standard Picard iteration as in It\^o's SDE theory on Euclidean spaces.

Let $\{W_i, i = 1,2,\cdots\}$ be i.i.d. Brownian motions, $\{\widetilde{A}_i, i = 1,2,\cdots\}$ and $U$ be smooth vector fields on $SG={\rm SDiff(M)}$. 
Suppose that the following SDE
 \begin{equation}\label{ch2-SDEG}
   \begin{cases}
    $$dg(s)=\sqrt{2\nu}\sum\limits_{i=1}^\infty \widetilde{A}_i(g(s))\circ dW_i(s)+U(s, g(s))ds,$$
    &\mbox{$s\in(t,T]$}\\
    $$g(t)=g$$ &\mbox{}
   \end{cases}
\end{equation}
 has a unique strong solution.

Let $e_x$ be the evaluation map, $e_x: SG\rightarrow M,~~ e_x(g)=g(x)$,
~~$de_x$ be the differential of $e_x$ at the point $g$ 
$$de_x: T_g(SG)\rightarrow T_{g(x)}M.$$ 
The following proposition gives the relation between $u(x)\in T_x M$ and $U\in T_e(SG)$.  It can be easily obtain by definition, and we omit the proof.
\begin{proposition}\label{U-u} Given any $U\in \mathcal{SG}=T_e(SG)$, we have
$$U(x)=(de_x)U, \ \ \forall x\in M.$$
Indeed, let $u(x)=(de_x)U$, we have
$$U(x)=u(x).$$
That is to say, the tangent vector $U\in\mathcal{SG}=T_e(SG)$ must be given by  $U=\{u(x): x\in M\}. $
\end{proposition}

By the chain rule of differential, we have 
$$d(e_x\circ R_g)_e=(de_x)_g(dR_g)_e.$$

\begin{lemma}
Suppose that there exists a family of vector fields $\{A_i\}_{i=1}^{\infty}$ on $M$ such that for all $g\in SG$, it holds 
$$A_i(g(x))=d(e_x\circ R_g)_e\widetilde{A}_i(e),$$
where $\widetilde{A}_i(e)\in\mathcal{SG}$ is the horizontal lift of $A_i(x)$ by $e_x$, and
$\widetilde{A}_i(g)$ is right-invariant, i.e., 
$$\widetilde{A}_i(g)=(dR_g)_e\widetilde{A}_i(e).$$
Then for any $f\in C^\infty(M,\mathbb{R})$, and $F_x=f\circ e_x$, it holds
$$(\mathcal{L}^{SG}F_x)(g)=(\mathcal{L}^Mf)(g(x)),$$
where  $\mathcal{L}^{SG}$ is the infinitesimal generator of SDE on $SG$
$$dg_t=\sum\limits_{i=1}^\infty \widetilde{A}_i(g_t)\circ dW_t^i+\widetilde{A}_0(g_t)dt, ~~g_0=g,$$
and $\mathcal{L}^M$ is the infinitesimal generator of SDE on $M$
$$dg_t(x)=\sum\limits_{i=1}^\infty A_i(g_t(x))\circ dW_t^i+A_0(g_t(x))dt, ~~g_0(x)=g(x).$$
\end{lemma}

\begin{proof} By the fact $A_i(g(x))=d(e_x\circ R_g)_e\widetilde{A}_i(e)=(de_x)_g(dR_g)_e\widetilde{A}_i(e)=(de_x)_g\widetilde{A}_i(g)$, we have
$\widetilde{A}_iF_x(g)=(A_if)(g(x))=(A_if)(e_x(g))=(A_if)\circ e_x(g)$. Hence, \begin{eqnarray*}
\widetilde{A}_iF_x&=&(A_if)\circ e_x, \\
\widetilde{A}_i(\widetilde{A}_iF_x)&=&A_i(A_if)\circ e_x.
\end{eqnarray*}
This implies
\begin{eqnarray*}
(\mathcal{L}^{SG}F_x)(g)&=&\left.\frac{\partial}{\partial t}\mathbb{E}_g[F_x(g_t)]\right|_{t=0}\\
&=&\sum\limits_{i=1}^\infty \widetilde{A}_i^2F_x(g)+\widetilde{A}_0F_x(g)\cdot\nabla^{SG}\\
&=&\sum\limits_{i=1}^\infty  A_i^2f(g(x))+A_0f(g(x))\cdot\nabla^M\\
&=&\left.\frac{\partial}{\partial t}\mathbb{E}_{g(x)}[f(g_t(x))]\right|_{t=0}\\
&=&(\mathcal{L}^Mf)(g(x)).
\end{eqnarray*}
\end{proof}

Fix $x\in M$. Let $e_x: SG\rightarrow M$ be the evaluation map, and $g(s)(x)=e_x(g(s))$. By the It\^o formula, we can derive the SDE for $g(s)(x)$ on $M$ from  SDE $(\ref{ch2-SDEG})$  for $g(s)$ on $SG$ as follows
 \begin{equation}\label{ch2-SDED}
   \begin{cases}
    $$dg(s)(x)=\sqrt{2\nu}\sum\limits_{i=1}^\infty A_i(g(s)(x))\circ dW_i(s)+U(s, g(s)(x))ds,$$
    &\mbox{$s\in(t,T]$}\\
    $$g(t)(x)=g(x).$$ &\mbox{}
   \end{cases}
\end{equation}
Here the vector fields $\widetilde{A}_i$ and $U$ on $SG = {\rm SDiff (M)}$ and  $A_i$ and $u$ on $M$ satisfy the relationship 
\begin{eqnarray*}
& &(de_x )(g)U(g) = u(g(x)),\\
& &(de_x )(g)\widetilde{A}_i (g)=A_i (g(x)) , i=1, 2, \cdots,
\end{eqnarray*}
which is the case when $\widetilde{A}_i$ and $U$ are right-invariant on $SG$ in the sense that 
\begin{eqnarray*}
& &U(g) = dR_g U(e),
~\widetilde{A}_i(g) = dR_g\widetilde{A}_i(e), \ \ i = 1, 2, \ldots\\
& &(de_x) U(e)=u(x),
~(de_x) \widetilde{A}_i(e) = A_i(x), \ \ i=1, 2,\ldots .
\end{eqnarray*}

Since $A_i$ and $u$ are $C^\infty$ vector field on $M$, then SDE $(\ref{ch2-SDED})$ has a unique 
solution $g(s)(x)$, and $x\mapsto g(s)(x)$ is a diffeomorphism on $M$ using the theory of stochastic flows. See \cite{IW} and \cite{Kunita:90}. The infinitesimal generator of $g(s)(x)$ of SDE $(\ref{ch2-SDED})$ is given by
\begin{eqnarray*}
\mathcal{L}^M=\nu \sum\limits_{i=1}^\infty  A_i^2+u\cdot \nabla^M.
\end{eqnarray*}

We need  a family of  vector fields $\{A_i, i\in \mathbb{N}\}$ 
such that 
\begin{eqnarray}
\sum\limits_{i=1}^\infty  A_i^2&=& \Delta^M, \label{cA1}\\
{\rm div} A_i&=&0, \ \ \forall i\in \mathbb{N},\label{cA2}
\end{eqnarray}
where 
$\Delta^M$ is the Laplace-Beltrami operator on the Riemannian manifold $(M, g)$.   When $M=\mathbb{T}^d$,  this was proved by Cipriano-Cruzeiro \cite{Cipriano-Cruzeiro:07}. When 
$M=\mathbb{S}^d$, this was proved by Le Jan-Raimond \cite{LR02}, see also Fang-Zhang \cite{Fang-Zhang:06}, 
Fang-Luo  \cite{Fang-Luo:15} and Luo \cite{Luo:15}. In  \cite{ELL},  Elworthy, 
Le Jan and X.-M. Li proved the existence of such vector fields  on any compact Riemannian manifold equipped with a Le Jan-Watanabe connection, in particular, for gradient Brownian system on any compact Riemannian manifold equipped with a torsion-skew symmetric metric connection. In particular, such vector fields exist on any compact Riemannian manifold with the Levi-Civita connection.  More precisely, assuming that there exists a family of  vector fields $\{A_i, i\in \mathbb{N}\}$ such that 
\begin{eqnarray}
&&|v|^2_{T_xM}=\sum\limits_{i=1}^\infty \langle  A_i(x), v\rangle^2, \ \ \forall x\in M, v\in T_xM,\label{ELL1}\\
&&\sum\limits_{i=1}^\infty \nabla_{A_i}A_i =0, \label{ELL2}\\
&&\sum\limits_{i=1}^\infty A_i\wedge \nabla_X A_i=0, \ \ \forall  X\in \Gamma(TM), \label{ELL3}
\end{eqnarray}
then  for any differential form $\omega$ on $M$, it holds (see  Remark 2.3.1 in \cite{ELL})
\begin{eqnarray}
\square^M \omega=-\sum\limits_{i=1}^\infty L_{A_i}^2\omega. \label{ELL4}
\end{eqnarray}

The main result of this subsection is the following  Hamilton-Jacobi-Bellman equation on $SG={\rm SDiff}(M)$.

\begin{theorem}\label{HJBGthm} Let $g(s)$ be the solution to the controlled SDE $(\ref{SDEG})$. Let 
\begin{eqnarray}
W(t, g)
:=\inf\limits_{U\in\mathcal{U}^{SG}}\mathbb{E}_{t, g}\left[\int_t^TL(s, g(s), U(s))ds+\Psi(g(T))\right] \label{W-Gvaluefunction}
\end{eqnarray}
be the value function of the related stochastic control problem, where  the Lagrangian is given by 
$$L(s, g(s), U(s))=\frac{1}{2}\|U\|_{T_{g(s)}(SG)}^2-\mathcal{V}(g(s)).$$
Then  $W$ satisfies the Hamilton-Jacobi-Bellman equation on $SG=\operatorname{SDiff}(M)$
\begin{equation}\label{HJB-W}
   \partial_tW-\frac{1}{2}\|\nabla^{SG}W\|^2_{T_g(SG)}+\nu \Delta^{SG}W=\mathcal{V}(g), 
   \end{equation}
   with the terminal condition $W(T, g)=\Psi(g)$. 
   \end{theorem}
   \begin{proof} 
By the extended Bellman dynamic programming principle on $SG$ (see Eq. $(\ref{HJB-4})$ in Section $3$), we have
\begin{eqnarray*}
\inf_{U\in\mathcal{SG}}\{\partial_tW+\langle\langle U, \nabla^{SG} W\rangle\rangle_{T_g (SG)}+\nu\Delta^{SG}W+L(t, g, U)\}=0.
\end{eqnarray*}
By the same argument as used in Section $2$ for the proof of the HJB equation on $M$, see also the description in Subsection $2.2$, we conclude that the optimal Markov control $U(t, g)$ is given by $$U(t, g)=-\nabla^{SG} W(t, g),$$ 
and $W$ satisfies the Hamilton-Jacobi-Bellman equation $(\ref{HJB-W})$  
      with the terminal condition $W(T, g)=\Psi(g)$. \end{proof}

\subsection{New derivation of the incompressible Navier-Stokes equation}

In this subsection, we use the Hamilton-Jacobi-Bellman equation on $SG=\operatorname{SDiff}(M)$ to give a new derivation of the incompressible Navier-Stokes equations on a compact Riemannian manifold $M$.  

Let $d^{SG}$ be the exterior differential on $SG$. 
Following \cite{ELL}, we introduce 
\begin{eqnarray}\label{divergence}
\delta^{SG}:=-\sum\limits_{i=1}^\infty {\rm int}_{\widetilde{A}_i} L_{\widetilde{A}_i}
\end{eqnarray}
and we define
\begin{eqnarray}\label{Hodge Laplace}
\square^{SG}:=d^{SG}\delta^{SG}+\delta^{SG}d^{SG}.
\end{eqnarray}
Since $(d^{SG})^2=0$, it holds
\begin{eqnarray*}
d^{SG}\square^{SG}&=&d^{SG}d^{SG}\delta^{SG}+d^{SG}\delta^{SG}d^{SG}\\
&=&d^{SG}\delta^{SG}d^{SG}\\
&=&d^{SG}\delta^{SG}d^{SG}+\delta^{SG}d^{SG}d^{SG}\\
&=&\square d^{SG}.
\end{eqnarray*}

Following \cite{ELL, Fang-Luo:15}, under \eqref{ELL1}, \eqref{ELL2}, \eqref{ELL3}, the Hodge like Laplacian acting on differential $1$-form $\omega$ on $SG$ has the following form
\begin{eqnarray}\label{Hodge}
\square^{SG} \omega=-\sum\limits_{i=1}^\infty L^2_{\widetilde{A}_i} \omega.
\end{eqnarray}
For a vector field $X$ and a differential $1$-form $\theta$, we denote by $X^b$ the dual $1$-form of $X$ and $\theta^\sharp$ the corresponding vector field of $\theta$. The Hodge like Laplacian acting on a vector field can be defined as follows
\begin{eqnarray}\label{HodgeV}
\square^{SG} X:=(\square^{SG} X^b)^\sharp.
\end{eqnarray}
By Fang-Luo \cite{Fang-Luo:15} under \eqref{ELL1}, \eqref{ELL2}, \eqref{ELL3} and the following  condition
\begin{eqnarray}\label{ELL4}
\sum_{i=1}^{\infty} ({\rm div} A_i) L_{A_i}=0,
\end{eqnarray}
the following holds
\begin{eqnarray}\label{HodgeVM}
\square^{M} X=-\sum\limits_{i=1}^\infty L^2_{A_i} X, \ \forall X\in \Gamma(TM).
\end{eqnarray}
However, the authors used the integration by parts formula in the proof to get \eqref{HodgeVM}. As we have not proved the integration by parts formula in the infinite dimensional group $SG={\rm SDiff}(M)$ of volume preserving  diffeomorphisms, we still do not know whether \eqref{HodgeVM} holds on  $SG={\rm SDiff}(M)$. 
To overcome this difficulty, we adopt the following definition of $\square^{SG}$ acting on vector fields on $SG$
\begin{eqnarray}\label{HodgeVG}
\square^{SG} X=-\sum\limits_{i=1}^\infty L^2_{\widetilde{A}_i} X, \ \forall X\in \Gamma(T(SG)).
\end{eqnarray}

Now we state the main result of this paper as follows.

\begin{theorem}\label{Main}  
There exists a one to one correspondence  between the set of smooth
solutions to the   incompressible Navier-Stokes equations with an external force  on a compact Riemannian manifold $M$ and the set of 
solutions to the viscous Burgers equation on $SG={\rm SDiff}(M)$. More precisely, $u(t, x)$ is a smooth solution to the backward incompressible Navier-Stokes equations (in brief,  NS)
on $M$
\begin{eqnarray}
\partial_t u+\nabla^{TM}_u u-\nu \square^M u&=&-\nabla^M p-v, \label{NSM-u1}\\
\rm{div}~\hspace{0.3mm}u&=&0,\label{NSM-u2}
\end{eqnarray}
where 
$$v(x)=(de_x)\nabla^{SG} \mathcal{V}(e)$$ 
is the divergence  free vector on $M$  corresponding to  $\nabla^{SG}\mathcal{V}(e)\in \mathcal{SG}$,
if and only if $U(t, g)=(dR_g) U(t, e)$ is a smooth solution to the following backward viscous Burgers equation on $SG={\rm SDiff}(M)$
\begin{eqnarray}
\partial_t U+\nabla^{T(SG)}_UU-\nu\square^{SG}U=-\nabla^{SG} \mathcal{V}, \label{BurgersG}
\end{eqnarray}
where $dR_g$ is the differential of the right translation  $R_g: SG\rightarrow SG$, and $U(t, e)=\{u(t, x): x\in M\}$.

 \medskip
 
Second, let $W\in C^{1,3}([0,T]\times SG)$ be a solution to the Hamilton-Jacobi-Bellman equation $(\ref{HJB-W})$ on $SG={\rm SDiff}(M)$, and let
\begin{eqnarray*}
U(t, g)=-\nabla^{SG} W(t, g)
\end{eqnarray*}
be the optimal Markov control, then $U(t, g)$ satisfies the backward viscous Burgers equation $(\ref{BurgersG})$ on $SG={\rm SDiff}(M)$. 

Finally, let   
$$u(t, x)=(de_x) U(t, e),$$
then there exists a function $p\in C^1([0, T]\times M, \mathbb{R})$ such that $(u, p)$ 
satisfies the backward incompressible Navier-Stokes equations $(\ref{NSM-u1})$ and  $(\ref{NSM-u2})$  with an external force  on $M$. 
\end{theorem}

\begin{proof} First, we prove the one to one correspondence  between the set of smooth
solutions to the   incompressible Navier-Stokes equations with an external force on a compact Riemannian manifold $M$ and the set of 
solutions to the viscous Burgers equation on $SG={\rm SDiff}(M)$.

\medskip
Let 
\begin{eqnarray*}
U(t, e)=\{ u(t, x): x\in M\}
\end{eqnarray*}
as in Theorem \ref{U-u}. Then $U(t, e)\in \mathcal{SG}$, 
and we have
\begin{eqnarray*}
\partial_t U(t, e)=\{\partial_t u(t, x): x\in M\}.
\end{eqnarray*}
 
By \cite{Arnold66, EM70}, the Riemannian metric
on $M$ induces the right-invariant $L^2$-Riemannian metric on $SG={\rm SDiff}(M)$. By Arnold \cite{Arnold66},  the geodesic equation on $SG={\rm SDiff}(M)$ gives the unique Levi-Civita $L^2$ connection $\nabla^{T(SG)}$ on $SG$. More precisely, by Arnold \cite{Arnold66} and Ebin-Marsden \cite{EM70},
\begin{eqnarray*}
(\nabla^{T(SG)}_U U)(e)=(\mathbb{P} (\nabla^{TM}_u u))^R
\end{eqnarray*}
where $\mathbb{P}$ is the Leray projection to the divergence free part and 
$X^R$ represents the right-invariant vector on $SG$ corresponding to vector field $X$ on $M$.
That is to say, there exists a smooth function $p_1\in C^1(M, \mathbb{R})$ such that 
\begin{eqnarray}
(\nabla^{T(SG)}_U U)(e)=\left\{\left(\nabla^{TM}_u u-\nabla^M p_1\right)(x): x\in M\right\}
\end{eqnarray}
and at any $g\in SG$, it holds
\begin{eqnarray}\label{RINabla}
\left(\nabla^{T(SG)}_UU\right)(g)=(dR_{g})(\nabla^{T(SG)}_U U)(e).
\end{eqnarray}

Note that 
\begin{eqnarray}
(de_x)[X, Y]^{SG}_e&=&[(de_x)X(e),(de_x)Y(e)]^M,\label{PushLie}\\
 (dR_g)[X, Y]^{SG}_e&=&[(dR_g)X(e),(dR_g)Y(e)]^{SG}_g,\label{RILie}
 \end{eqnarray}
where $[X,Y]$ represents the Lie bracket of vector fields $X$ and $Y$.
Since $L_XY=[X,Y]$, by \eqref{Hodge},
we have
\begin{eqnarray*}
(de_x)(\square^{SG} U)(e)&=&-\sum_i (de_x) (L^2_{\widetilde{A}_i}U)(e)\\
&=&-\sum_i (de_x) (L_{\widetilde{A}_i}(L_{\widetilde{A}_i}U))(e)\\
&=&-\sum_i  L_{A_i}(de_x) (L_{\widetilde{A}_i}U)(e)\\
&=&-\sum_i  L_{A_i} (L_{A_i}u)(x)\\
&=&-\sum_i  L_{A_i}^2u(x).
\end{eqnarray*}
By Fang-Luo \cite{Fang-Luo:15}, it holds

\begin{eqnarray*}
\square^M u=-\sum\limits_i   L_{A_i}^2u.
\end{eqnarray*}
Therefore,
\begin{eqnarray*}
\square^{SG} U(t, e)=\{\square^M u(t, x): x\in M\}.
\end{eqnarray*}

\medskip

In view of this,  we have
\begin{eqnarray*}
\left(\partial_t U+\nabla^{T(SG)}_U U+\nu \square^{SG} U\right)(t, e)=\{\left(\partial_t u+\nabla^{TM}_u u-\nu \square^M u -\nabla p_1\right)(t, x): x\in M\}.
\end{eqnarray*}
Let  $X\in \mathcal{SG}$. Then, by the right-invariant property of the inner product 
$\langle\langle \cdot, \cdot\rangle\rangle$, \eqref{RINabla}, \eqref{RILie}  and $\partial_tU(t, g)=(dR_g)\partial_tU(t, e)$,
we obtain
\begin{eqnarray}
\left\langle\left\langle\partial_t U, X \right\rangle\right\rangle_{T_g(SG)}&=&
\left\langle\left\langle\partial_t U, X \right\rangle\right\rangle_{T_e(SG)}\nonumber\\
&=&
\int_M\langle\partial_t u(t, x), X(x)\rangle dv(x), \label{DSU1}\\
\langle\langle \square^{SG} U, X\rangle\rangle_{T_g(SG)}&=&\langle\langle \square^{SG} U, X\rangle\rangle_{T_e(SG)}\nonumber\\
&=&\int_M \langle (\square^M u)(t, x), X(x) \rangle dv(x), \label{DSU2}\\
\left\langle\left\langle  \nabla^{T(SG)}_U U, X \right\rangle\right\rangle_{T_g(SG)}
&=&
\left\langle\left\langle  \nabla^{T(SG)}_U U, X \right\rangle\right\rangle_{T_e(SG)}
\nonumber\\
&=&\int_M \left\langle (\nabla^{TM}_u u -\nabla p_1)(t, x),  X(x)\right\rangle dv(x).\label{DSU3}
\end{eqnarray}
For the calculation of $\nabla^{SG} \mathcal{V}(g)$, see Proposition \ref{prop gradient Diff}.
By Proposition \ref{U-u}, $$\nabla^{SG} \mathcal{V}(e)=\mathbb{P}(\nabla^{G} \widetilde{\mathcal{V}}(e))=\{v(x): x\in M\},$$
where $v(x)=(de_x)\nabla^{SG}\mathcal{V}(e)$.
Then 
\begin{eqnarray}\label{DSU4}
\langle\langle \nabla^{SG}\mathcal{V}, X \rangle\rangle_{T_g(SG)}=
\langle\langle \nabla^{SG}\mathcal{V}, X \rangle\rangle_{T_e(SG)}=
\int_M \langle v(x), X(x) \rangle dv(x).
\end{eqnarray}
\medskip

Now, suppose that  $U(t, g)$ is a solution to the backward viscous Burgers equation \eqref{BurgersG} on $SG$. Then for any  $X\in \mathcal{SG}=T_e(SG)$,  
combining $(\ref{DSU1})$,  $(\ref{DSU2})$, $(\ref{DSU3})$ and $(\ref{DSU4})$ with $(\ref{BurgersG})$, we have
\begin{eqnarray*}
0&=&\left\langle\left\langle \partial_t U(t, g)+\nabla^{T(SG)}_U U(t, g)-\nu\square^{SG}U(t, g)+\nabla^{SG} \mathcal{V}(g),  X( g) \right\rangle\right\rangle\\
&=&\int_M \left\langle \left(\partial_t u+\nabla^{TM} _u u-\nabla p_1-\nu \square^M u\right)(t, x)
+v(x), X(x)\right\rangle
 dv(x).
\end{eqnarray*}
That is to say, $\left(\partial_t u+\nabla^{TM} _u u-\nabla p_1-\nu \square^M u \right)(t, x)+v(x)$ is orthogonal to all 
$X\in \mathcal{SG}=\{X\in \Gamma(TM): {\rm div}\hspace{0.3mm}X=0\}$. By the Hodge decomposition theorem on compact Riemannian manifolds \cite{Taylor}, there exists a function $p\in C^1([0, T]\times M)$ such that $u=u(t, x)$ satisfies the backward 
incompressible Navier-Stokes equations with an external force  on M
\begin{equation*}
    \begin{cases}
    $$ \partial_tu+\nabla^{TM}_u u-\nu \square^M u=-\nabla^M p-v,$$&\mbox{}\\
    $$\rm{div}\hspace{0.3mm}~ u=0,$$&\mbox{}
    \end{cases}
\end{equation*}
where  $p_1$ is absorbed into $p$. 

This proves that: ${\rm viscous \ Burgers\ equation\ on\ SG} \Longrightarrow \ {\rm NS\ equations}$ with  an external force  on $M$.

Conversely, if $(u(t, x), p(t, x))$ is a solution to backward NS
with an external force, then  for any $X\in \mathcal{SG}=\{X\in \Gamma(TM): {\rm div}\hspace{0.3mm}X=0\}$, we have
\begin{eqnarray*}
& &\left\langle\left\langle \partial_t U(t, g)+\nabla^{T(SG)}_U U(t, g)-\nu\square^{SG}U(t, g)+\nabla^{SG} \mathcal{V}(g),  X(g) \right\rangle\right\rangle\\
&=&\int_M \left\langle \left(\partial_t u+\nabla^{TM} _u u-\nabla^M p_1-\nu \square^M u\right)(t, x)+v(x), X(x)\right\rangle
 dv(x)\\
 &=&\int_M \left\langle -\nabla^M p_1(t, x)-\nabla^M p(t, x), X(x)\right\rangle
 dv(x)\\
 &=&\int_M \left\langle  p_1(t, x)+p(t, x), {\rm div}~X(x)\right\rangle
 dv(x) \\
 &=&0 . 
\end{eqnarray*}
This proves that $U(t, g)$ is a solution to the backward viscous Burgers equation \eqref{BurgersG} on $SG$. This proves that: ${\rm NS\ with \ an \ external \ force \ on\ M} \Longrightarrow {\rm viscous \ Burgers\ on\ SG}$. Hence, we prove that: ${\rm NS\ with \ an \ external \ force \ on\ M} \Longleftrightarrow {\rm viscous \ Burgers\ on\ SG}$. Indeed, when $\nu=0$, this result is essentially the Arnold's famous theorem on the equivalence between the $L^2$-geodesic equation on $SG$ and the incompressible Euler equation with an external  force\  on $M$ (see \cite{Arnold}).

\medskip

Next, we prove that: ${\rm HJB \ on \ SG}\Longrightarrow {\rm viscous\  Burgers\ on\ SG}$. More precisely, if $W$ is a smooth solution to the Hamilton-Jacobi-Bellman equation $(\ref{HJB-W})$, then $U=-\nabla^{SG} W$ satisfies the backward viscous Burgers equation $(\ref{BurgersG})$. Indeed, taking differentiation on the Hamilton-Jacobi-Bellman equation $(\ref{HJB-W})$, we have 
\begin{eqnarray*}
d^{SG}(\partial_tW)-d^{SG}\left(\frac{1}{2}\|\nabla^{SG} W\|^2_{T_g(SG)}\right)+\nu d^{SG}(\Delta^{SG}W)=d^{SG} \mathcal{V}(g).
\end{eqnarray*}
By Carten's formula, 
$$L_{\widetilde{A}}=d^{SG}i_{\widetilde{A}}+i_{\widetilde{A}}d^{SG},$$
where $\widetilde{A}$ is a vector field on $SG$, and $(d^{SG})^2=0$, we have 
$$d^{SG}L_{\widetilde{A}}=L_{\widetilde{A}}d^{SG}.$$
Using the commutative formulae
$$d^{SG} \partial_t W=\partial_t d^{SG} W, \ d^{SG} \Delta^{SG} W=d^{SG} \sum_i L^2_{\widetilde{A}_i}W=\sum_i  L^2_{\widetilde{A}_i}d^{SG}W=
 -\square^{SG} d^{SG} W,$$ we have
\begin{eqnarray}
\partial_t d^{SG}W-\frac{1}{2}d^{SG} \|\nabla^{SG} W\|^2_{T_g(SG)}-\nu\square^{SG} d^{SG}W=d^{SG} \mathcal{V}(g). \label{HJBW}
\end{eqnarray}
Thus, since $\nabla^{SG} \|U\|^2=2\nabla^{T(SG)}_U U$, we obtain that $U(t, g)=-\nabla^{SG} W(t, g)$ satisfies the backward viscous Burgers equation $(\ref{BurgersG})$ on $SG$
\begin{eqnarray*}
\partial_t U+\nabla^{T(SG)}_UU-\nu \square^{SG} U=-\nabla^{SG} \mathcal{V}.
\end{eqnarray*}
Indeed, by 
the duality between the Riemmanian  gradient and exterior differential operator on $\mathcal{SG}=\{X\in \Gamma(TM): \rm{div}\hspace{0.3mm} X=0\}$,  for 
any  $X\in \mathcal{SG}=T_e(SG)$,  we have
\begin{eqnarray*}
0&=&\left(\partial_t d^{SG}W-\frac{1}{2}d^{SG} \|\nabla^{SG} W\|^2_{T_g(SG)}-\nu \square^{SG} d^{SG}W-d^{SG} \mathcal{V}\right)(X)\\
&=&\left\langle\left\langle \partial_t\nabla^{SG}W-\nabla^{SG}\left(\frac{1}{2}\|\nabla^{SG} W\|^2_{T_g(SG)}\right)-\nu\square^{SG}\nabla^{SG}W-\nabla^{SG} \mathcal{V},  X \right\rangle\right\rangle\\
&=&\left\langle\left\langle -\partial_tU-\nabla^{SG}\left(\frac{1}{2}\|U\|^2_{T_g(SG)}\right)+\nu\square^{SG} U-\nabla^{SG} \mathcal{V},  X \right\rangle\right\rangle.
\end{eqnarray*}

\medskip
Finally,  as the consequence of the first and second assertations,   
we conclude that 
$$u(t, x)=de_x U(t, e)=-(de_x)\nabla^{SG} W(t, e)$$ is a solution to the backward incompressible Navier-Stokes equations $(\ref{NSM-u1})$ and $(\ref{NSM-u2})$ with an external  force. 
\end{proof}

\begin{remark} \label{forwardNS} By time reversal, $(-u(T-t, \cdot), p(T-t, \cdot))$ satisfies the forward incompressible Navier-Stokes equation with an external force on $M$
\begin{eqnarray}\label{NSMF1}
\partial_t u+\nabla^{TM}_u u+\nu \square^M u&=&-\nabla^M p-v, \\
\rm{div}\hspace{0.3mm}~u&=&0.\label{NSMF2}
\end{eqnarray}

Moreover, 
in the case $\mathcal{V}(g)=\int_MV(g(x))dx, \  \forall g\in SG={\rm SDiff}(M)$
for some $V: M\rightarrow \mathbb{R}$, 
by  \eqref{v0} in Proposition \ref{prop gradient Diff}  below, we have
$v(x)=(de_x)\nabla^{SG}\mathcal{V}(e)=0$.
Therefore, \eqref{NSMF1}, \eqref{NSMF2} is actually the forward incompressible Navier-Stokes equation.
\end{remark} 

\begin{remark}
By the right-invariant of the vector field $\nabla^{SG} W(t, g)$ on $SG$, we have 
\begin{eqnarray*}
\nabla^{SG} W(t, g)=(dR_g)\nabla^{SG} W(t, e),
\end{eqnarray*}
which yields
\begin{eqnarray*}
u(t, x)&=&-\left(d e_{g^{-1}(x)}\right)\nabla^{SG} W(t, g)\\
&=&-\left(d (e_x\circ R_{g^{-1}})\right)\nabla^{SG} W(t, g)\\
&=&-(d e_x)\circ (dR_{g^{-1}})\nabla^{SG} W(t, g)\\
&=&-(de_x) \nabla^{SG} W(t, e),
\end{eqnarray*}
where $g^{-1}$ is the inverse of $g$.
\end{remark}

\begin{remark} Let $W(t, g)$ be a solution to HJB equation \eqref{HJB-W},  then $U(t, g)=-\nabla^{SG} W(t, g)$ satisfies the viscous Burgers equation \eqref{BurgersG}. 
Conversely, if the initial value is ``gradient type'', i.e. $U(0, g)=-\nabla^{SG} \phi(g)\in C^{1}(SG, T(SG))$ for some $\phi \in C^{2}(SG, \mathbb{R})$, it is interesting to ask that whether the solution to viscous Burgers equation \eqref{BurgersG}
is also ``gradient type'', i.e. there exists a function $W\in C^{1,2}([0,T]\times SG, \mathbb{R})$ with $W(0, g)=\phi(g)$ such that $U(t, g)=-\nabla^{SG} W(t, g)\in C^{1,1}([0,T]\times SG, T(SG))$ for all $g\in SG$. Moreover,  $W(t, g)$ satisfies the HJB equation \eqref{HJB-W}. 
 In finite dimensional Euclidean spaces, it is a well known result, see \cite{Cole, Hopf, Taylor}.
In infinite dimensional group $SG={\rm SDiff}(M)$ of volume preserving diffeomorphsims, we can prove that $U(t, g)=-\nabla^{SG} W(t, g)$ satisfies the viscous Burgers equation with ``gradient type'' Cauchy initial value, but the uniqueness of the solution to the viscous Burgers equation with ``gradient type'' Cauchy initial value is in our future research plan.

For viscous Burgers equation in finite dimensional Euclidean space or Riemannian manifold, the local existence and uniqueness of smooth solutions can be find in \cite{Taylor}, in particular,  the solution in suitable Sobolev spaces,
see Theorem 4.1, P558 in\cite{LSU} and  the result on torus in \cite{PR}. It is interesting to ask that whether the theorem of local existence and uniqueness of smooth solutions holds for the Cauchy problem of  the viscous Burgers equation \eqref{BurgersG} on the infinite dimensional group $SG={\rm SDiff}(M)$ of volume preserving diffeomorphsims.
This is in our future research plan.
\end{remark}

 \begin{remark}  
Here we would like to point out that  
even  the vector field $U(t, g)=-\nabla^{SG} W(t, g)$ is a gradient type vector field on $SG={\rm SDiff}(M)$, and it could be regarded as an ``irrotational'' vector field on 
$ SG={\rm SDiff}(M)$,
the solutions to the incompressible Navier-Stokes equations, i.e. 
$$ u(t, x)=-\left(d e_{x}\right)\nabla^{SG} W(t, e),$$
which is  the pushing forward vector  to the compact Riemannian manifold $M$ by the differential of the evaluation map,  is a divergence free vector and  is not a gradient type vector on $M$. 
In particular, in the case ${\rm dim}~M=3$, $u(t, x)=-\left(d e_{x}\right)\nabla^{SG} W(t, e)$ is not an irrotational vector  on $M$. 
\end{remark}

\medskip

To better explain  $\nabla^{SG}W$, we give an explicitly calculation of $u(t, x)=-(de_x) \nabla^{SG} W(t, e)$ in the following Proposition.
 $i: SG={\rm SDiff}(M)\rightarrow G={\rm Diff}(M)$ be the canonical isometric mapping,
$\mathcal{V}:SG={\rm SDiff}(M)\rightarrow \mathbb{R}$ and $\widetilde{\mathcal{V}}:{G=\rm Diff}(M)\rightarrow \mathbb{R}$ satisfying $\widetilde{\mathcal{V}}(g)=\mathcal{V}(g), \forall g\in SG={\rm SDiff}(M)$. Let $\mathbb{P}: L^2(M, \Gamma(TM))\rightarrow L^2(M, \Gamma(TM))\cap Ker({\rm div})$ be the Leray orthogonal projection. Since $i: SG\rightarrow G$ is isometry, we obtain
$$\langle\langle X,Y\rangle\rangle_{SG}=\langle\langle X,Y\rangle\rangle_{G}, \forall X, Y\in T_g(SG).$$
By definition,
$$\langle\langle \nabla^{G} W(g), X(g) \rangle\rangle=dW(g) X(g)=\left.\frac{\partial}{\partial \varepsilon}\right|_{\varepsilon=0}W(e^{\varepsilon X}g), \forall g\in G, ~X(g)\in T_g{G}.$$

\begin{proposition}\label{prop gradient Diff}
Let $\widetilde{\mathcal{V}}(g)=\int_MV(g(x))dv(x), \forall g\in G={\rm Diff}(M)$ for some $V: M\rightarrow \mathbb{R}$. Then
the gradient of $\widetilde{\mathcal{V}}$  on $G={\rm Diff}(M)$ is given by
$$(de_{g^{-1}(x)})(\nabla^{G}\widetilde{\mathcal{V}}(g))=\frac{\nabla^M V(x)}{det(Dg(g^{-1}(x)))}.$$
Furthermore,  we have
\begin{eqnarray}\label{v0}
v(x)=(de_x)\nabla^{SG}\mathcal{V}(e)=0, 
\end{eqnarray}
where 
$\mathcal{V}(g)=\int_MV(g(x))dv(x), \forall g\in SG={\rm SDiff}(M)$.

Let $\mathcal{V}$ be a smooth cylinder function on $SG$, that is, there exists a smooth 
 function $V\in C^\infty(M, \mathbb{R})$ such that 
$$\mathcal{V}(g)=V(g(x_1), \cdots, g(x_n) ),\ \ \ \forall g\in SG,\ \ \ \ {\rm for \ fixed} \ \ x_1,\cdots, x_n\in M.$$
Then the gradient of 
a function on $SG={\rm SDiff}(M)$ is given by 
\begin{eqnarray*}
(de_{g^{-1}(x)})\nabla^{SG}\mathcal{V}(g)=\frac{1}{vol(M)} \sum_{i=1}^n \mathbb{P}\left(\nabla^M_{y_i} V(\cdot)\delta_{x_i}(g^{-1}(\cdot))\right)(x),
\end{eqnarray*}
where $\nabla_{y_i}^M V$ is the gradient of $y_i\mapsto V(\cdot, y_i, \cdot)$ with respect to the $i$-th variable $y_i$ on $M$.

For general form of function $\mathcal{V}$ on $SG$,  we have
\begin{eqnarray*}
\nabla^{SG} \mathcal{V}= \mathbb{P} \nabla^{G} \widetilde{\mathcal{V}}=\sum\limits_{i}  \langle \langle \nabla^{G} \widetilde{\mathcal{V}}, \widetilde{A}_i\rangle \rangle \widetilde{A}_i,
\end{eqnarray*}
where $\{\widetilde{A}_i, i\in \mathbb{N}\}$ is the horizontal lift of a
complete orthonormal basis of divergence free vector fields $\{A_i, i\in \mathbb{N}\}$ on $M$.
\end{proposition}

\begin{proof}
If $\widetilde{\mathcal{V}}(g)=\int_MV(g(x))dv(x), \forall g\in G={\rm Diff}(M)$, by the change of variable, we have
\begin{eqnarray*}
\widetilde{\mathcal{V}}(g)=\int_MV(g(x))dv(x)=\int_MV(y)\frac{dv(y)}{det(Dg(g^{-1}(y)))}.
\end{eqnarray*}
Then
\begin{eqnarray}\label{gradient Diff 1}
\langle \langle  \nabla^{G}\widetilde{\mathcal{V}}(g), X(g)\rangle\rangle
&=&\left.\frac{\partial}{\partial \varepsilon}\right|_{\varepsilon=0}\widetilde{\mathcal{V}}(e^{\varepsilon X}g) \nonumber\\
&=&\left.\frac{\partial}{\partial \varepsilon}\right|_{\varepsilon=0}\int_MV((e^{\varepsilon X}g)(x))dv(x) \nonumber\\
&=&\int_M(dV)(g(x)) X(g(x))dv(x) \nonumber\\
&=&\int_M\langle \nabla^M V(g(x)), X(g(x)) \rangle dv(x)\nonumber\\
&=&\int_M\langle \nabla^M V(y), X(y) \rangle \frac{dv(y)}{det(Dg(g^{-1}(y)))}.
\end{eqnarray}
Hence,
$$(de_{g^{-1}(x)})(\nabla^{G}\widetilde{\mathcal{V}}(g))=\frac{\nabla^M V(x)}{det(Dg(g^{-1}(x)))}.$$
Furthermore, since $g\in SG$, we have  $det(Dg)=1$. Hence, by right-invariant of the Riemannian metric, we obtain
$$
\left\langle\left\langle \nabla^{SG}\mathcal{V}, X \right\rangle\right\rangle=\int_M\langle \nabla^M V(x), X(x) \rangle dv(x)=0, 
\forall X\in T_e(SG).
$$
That is to say,  $v(x)=(de_x)\nabla^{SG} \mathcal{V}(e)$ is orthogonal to all 
$X\in \mathcal{SG}=\{X\in \Gamma(TM): {\rm div}\hspace{0.3mm}X=0\}$. 
On the other hand, ${\rm div}~v=0$. We obtain that $v
=(de_x)\nabla^{SG}\mathcal{V}(e)=0.$

If $\widetilde{\mathcal{V}}(g)=V(g(x_1),\cdots, g(x_n))$ is a cylinder function on $G={\rm Diff}(M)$, we have
\begin{eqnarray*}
\langle\langle \nabla^{G} \widetilde{\mathcal{V}}(g), X(g) \rangle\rangle&=&\left.\frac{\partial}{\partial \varepsilon}\right|_{\varepsilon=0}\widetilde{\mathcal{V}}(e^{\varepsilon X}g),\\
&=&\left.\frac{\partial}{\partial \varepsilon}\right|_{\varepsilon=0}V(e^{\varepsilon X}g(x_1),\cdots,
e^{\varepsilon X}g(x_n)),\\
&=&\sum_{i=1}^n \left\langle (\nabla^M_{y_i}V)(g(x_i)),X(g(x_i))\right\rangle\\
&=&\frac{1}{vol(M)}\int_M \sum_{i=1}^n \left\langle (\nabla^M_{y_i}V)(g(x))\delta_{x_i}(x),X(g(x))\right\rangle dv(x)\\
&=&\frac{1}{vol(M)}\int_M \sum_{i=1}^n \left\langle (\nabla^M_{y_i}V)(y)\delta_{x_i}(g^{-1}(y)), X(y)\right\rangle \frac{dv(y)}{det(Dg(g^{-1}(y)))}
\end{eqnarray*}
where $\nabla_{y_i}^M V$ is the Riemannian gradient of $y_i\mapsto V(\cdot, y_i, \cdot)$ with respect to the $i$-th variable $y_i$ on $M$, and in the last step we have changed of variable. Then we obtain 
\begin{eqnarray*}
(de_{g^{-1}(x)})\nabla^{G}\widetilde{\mathcal{V}}(g)=\frac{1}{vol(M)} \sum_{i=1}^n\frac{\nabla^M_{y_i} V(x)\delta_{x_i}(g^{-1}(x))}{det(Dg(g^{-1}(x)))}.
\end{eqnarray*}
Hence, 
\begin{eqnarray*}
(de_{g^{-1}(x)}) \nabla^{SG}\mathcal{V}(g)=(de_{g^{-1}(x)}) \mathbb{P}(\nabla^{G} \widetilde{\mathcal{V}}(g))=\frac{1}{vol(M)} \sum_{i=1}^n\mathbb{P}\left(\frac{\nabla^M_{y_i} V(\cdot )\delta_{x_i}(g^{-1}(\cdot))}{det(Dg(g^{-1}(\cdot)))}\right)(x).
\end{eqnarray*}
In particular, for $g\in SG={\rm SDiff}(M)$, since $det(Dg(g^{-1}(x)))=1$, we have
\begin{eqnarray*}
(de_{g^{-1}(x)}) \nabla^{SG}\mathcal{V}(g)&=&\frac{1}{vol(M)} \sum_{i=1}^n\mathbb{P}\left(\nabla^M_{y_i} V(\cdot)\delta_{x_i}(g^{-1}(\cdot))\right)(x).
\end{eqnarray*}
Taking $g=e$, we have
\begin{eqnarray*}
(de_{x}) \nabla^{SG}\mathcal{V}(e)&=&\frac{1}{vol(M)} \sum_{i=1}^n\mathbb{P}\left(\nabla^M_{y_i} V(\cdot)\delta_{x_i}(\cdot)\right)(x).
\end{eqnarray*}

Hence, $u(t, x)=-(de_x)\nabla^{SG}W(t, e)$ is a divergence free vector field on $M$, and is not a gradient type vector field on $M$.
\end{proof}
\begin{example}
Consider the n-dimensional torus $M=\mathbb{T}^n$, see \cite{Cipriano-Cruzeiro:07, Fang-Luo:15}, 
for $k\in \mathbb{T}^n/\{0\}$, let  $\{e_{k, j}, j=1, \cdots, n-1\}$ be the orthonormal basis of $k^{\perp}$, then
  $\{A_{k, j}(\theta)=\frac{2}{(2\pi)^n}\cos(k\cdot \theta)e_{k, j}, B_{k, j}(\theta)=\frac{2}{(2\pi)^n}\sin(k\cdot \theta)e_{k, j}, k\in \mathbb{T}^n/\{0\}, j=1, \cdots, n-1 \}$ is a complete orthonormal basis of divergence free vector fields $X$ in $L^2(\mathbb{T}^n)$ such that $\int_{\mathbb{T}^n}  X(\theta) d\theta =0$.  By \cite{Cipriano-Cruzeiro:07, Fang-Luo:15}, one can choose renormalization constants $\{c_k, k=1,\cdots,\infty\}$ such that SDE \eqref{SDEG} with $M=\mathbb{T}^d$ has a unique strong solution. Since
\begin{eqnarray*}
&& (\nabla^M_{y_1} V,\cdots,\nabla^M_{y_n} V)\\
&=&\sum_{k\in \mathbb{Z}^n/\{0\}}\sum_{j=1}^{n-1} c_k^2[ \langle\langle (\nabla^M_{y_1} V,\cdots,\nabla^M_{y_n} V), \widetilde{A}_{k, j}(\theta) \rangle\rangle  \widetilde{A}_{k, j}(\theta)\\
&&+\langle\langle(\nabla^M_{y_1} V,\cdots,\nabla^M_{y_n} V), \widetilde{B}_{k, j}(\theta)\rangle \rangle \widetilde{B}_{k, j}(\theta)]
\end{eqnarray*}
we obtain
\begin{eqnarray*}
 \sum_{i=1}^n\mathbb{P}\left(\nabla^M_{y_i} V(\cdot)\delta_{x_i}(\cdot)\right)
&=&\sum_{i=1}^n\sum_{k\in \mathbb{Z}^n/\{0\}}\sum_{j=1}^{n-1}  c_k^2[ \langle\langle \nabla^M_{y_i} V, \widetilde{A}_{k, j}(\theta) \rangle\rangle  \widetilde{A}_{k, j}(\theta)\\
&&+\langle\langle\nabla^M_{y_i} V, \widetilde{B}_{k, j}(\theta)\rangle \rangle \widetilde{B}_{k, j}(\theta)].
\end{eqnarray*}
 We need the renormalization constants $\{c_k, k=1,\cdots,\infty\}$ such that $U(t, e)=-\nabla^{SG} W(t, e)$ takes value in $Q^{1\over 2} (L^2(M))$, where 
$Q_{ij}=c_i\delta_{ij}, i, j=1,\cdots, \infty$.
\end{example}

\begin{example}
Consider the n-dimensional sphere $M=\mathbb{S}^n$, see \cite {LR02, Fang-Zhang:06, 
Fang-Luo:15, Luo:15},
 let $d_l$ be the dimensional of the divergence free eigenvector  spaces of $\square^M$ and let $\{V_{l, k}, l=1,\cdots,\infty, k=1, \cdots , d_l\}$ be a complete orthonormal basis of the divergence free eigenvector  spaces. By \cite{LR02, Fang-Zhang:06, 
Fang-Luo:15, Luo:15}, one can choose renormalization constants $\{c_l, l=1,\cdots,\infty\}$ such that SDE \eqref{SDEG} with $M=\mathbb{S}^n$ has a unique strong solution.
Since
\begin{eqnarray*}
&&(\nabla^M_{y_1} V,\cdots,\nabla^M_{y_n} V)\\
&=&\sum_{l=1}^{+\infty}\sum_{k=1}^{d_l} c_l^2 \langle\langle(\nabla^M_{y_1} V,\cdots,\nabla^M_{y_n} V), \widetilde{V}_{l, k} \rangle\rangle  \widetilde{V}_{l, k},
\end{eqnarray*}
we obtain
$$ \sum_{i=1}^n\mathbb{P}\left(\nabla^M_{y_i} V(g(x_i))\right)= 
\sum_{i=1}^n\sum_{l=1}^{+\infty}\sum_{k=1}^{d_l}c_l^2 \langle \langle\nabla^M_{y_i} V, \widetilde{V}_{l, k} \rangle \rangle \widetilde{V}_{l, k}.$$
We need the renormalization constants $\{c_l, l=1,\cdots,\infty\}$ such that $U(t, e)=-\nabla^{SG} W(t, e)$ takes value in $Q^{1\over 2} (L^2(M))$, where 
$Q_{ij}=c_i\delta_{ij}, i, j=1,\cdots, \infty$.
\end{example}

\medskip

In summary, we have the following graph which illustrates the relationship between the Hamilton-Jacobi-Bellman and viscous Burgers equation on $SG={\rm SDiff}(M)$ 
\begin{eqnarray*}
{\rm Controlled \ SDE\ (\ref{SDEG})\ on}\  SG={\rm SDiff}(M) &\Longrightarrow &{\rm DPP\ on}\ SG\Longrightarrow  {\rm HJB \ Eq} \ (\ref{HJB-W}) \ {\rm on}\ SG\\
{\rm differentiating\ HJB\ Eq\ (\ref{HJB-W}) \ on}\ SG  &\Longrightarrow & {\rm viscous \ Burgers\ Eq}\ (\ref{BurgersG}) \ {\rm on}\ SG.
\end{eqnarray*}
Moreover, we have the following equivalence between the viscous Burgers equation on $SG={\rm SDiff}(M)$ and and the incompressible Navier-Stokes equations  on $M$
\begin{eqnarray*}
&&{\rm \ pushing\ forward}\ {\rm Burgers\ Eq}\ (\ref{BurgersG})\ to \ M \ via \ de_{x}: T_e(SG) \rightarrow T_x M\ \\
&&\Longleftrightarrow   {\rm NS\ Eqs\ \eqref{NSM-u1}\ \eqref{NSM-u2} \ with\ an \ external\ force\  on}\ M.
\end{eqnarray*}

\medskip

\section{The incompressible Euler equation}

In this section, we consider the particular case where the viscosity vanishes, i.e., $\nu=0$, and derive  the incompressible Euler equation from the deterministic dynamic programming principle on the group $SG={\rm SDiff}(M)$ over a compact Riemannian manifold.

Consider the deterministic control problem on $SG={\rm SDiff}(M)$
\begin{eqnarray*}
W(t, g)
:=\inf\limits_{U\in\mathcal{U}^{SG}}\left[\left.\int_t^TL(s, g(s), U(s))ds+\Psi(g(T))\right|g(t)=g\right],
\end{eqnarray*}
where $g(s)$ satisfies the following ODE on $SG$ 

\begin{equation}\label{ODEG}
   \begin{cases}
    $${d\over ds} g(s)=U(s, g(s)),$$
    &\mbox{$s\in(t,T],$}\\
    $$g(t)=g.$$ &\mbox{}
   \end{cases}
\end{equation}
Consider  the Lagrangian 
$$L(s, g(s), U(s))=\frac{1}{2}\|U\|_{T_{g(s)}(SG)}^2-\mathcal{V}(g(s)).$$
By the extended deterministic dynamic programming principle on $SG={\rm SDiff}(M)$, whose proof can be given similarly to the proof of Theorem \ref{DPPG}, we have
\begin{eqnarray*}
\inf_{U\in T_g(SG)}\{\partial_tW+\langle \langle U, \nabla^{SG}W\rangle \rangle_{T_g(SG)}+L(t, g, U)\}=0.
\end{eqnarray*}
The optimal Markov control $U(t, g)$ is of the form 
$$U(t, g)=-\nabla^{SG} W(t, g),$$ 
and  the value function $W$ satisfies the Hamilton-Jacobi equation
\begin{equation}\label{HJ-W}
   \partial_tW-\frac{1}{2}\|\nabla^{SG}W\|^2_{T_g(SG)}=\mathcal{V}, 
\end{equation}
with the terminal value
   $$W(T, g)=\Psi(g).$$
   
   \vskip0.5cm
   
The  following result gives a new derivation of the incompressible Euler equation with an external force from the Hamilton-Jacobi equation on $SG={\rm SDiff}(M)$.
\begin{theorem}\label{Main-Euler}  There exists a one to one correspondence  between the set of smooth
solutions to the   incompressible Euler equation with an external force on a compact Riemannian manifold $M$ and the set of 
solutions to the Burgers equation on $SG={\rm SDiff}(M)$. More precisely, $u(t, x)$ is a smooth solution to the  incompressible Euler equation with an external force
on $M$
\begin{eqnarray}
\partial_t u+\nabla^{TM}_u u&=&-\nabla^M p-v, \label{EME-1}\\
{\rm div}~u&=&0, \label{EME-2}
\end{eqnarray}
where 
$$v(x)=(de_x)\nabla^{SG} \mathcal{V}(e)$$ 
is the divergence free vector on $M$ corresponding to  $\nabla^{SG}\mathcal{V}(e)\in T_e(SG)$,
if and only if $U(t, g)=(dR_g) U(t, e)$ is a smooth solution to the following Burgers equation on $SG={\rm SDiff}(M)$
\begin{eqnarray}
\partial_t U+\nabla^{T(SG)}_UU=-\nabla^{SG} \mathcal{V}, \label{BurgersG2}
\end{eqnarray}
where $dR_g$ is the differential of the right translation  $R_g: SG\rightarrow SG$, and $U(t, e)=\{u(t, x): x\in M\}$.

 \medskip
 
Second, let $W\in C^{1,2}([0,T]\times SG)$ be a solution to the Hamilton-Jacobi equation $(\ref{HJ-W})$ on $SG={\rm SDiff}(M)$, and let
\begin{eqnarray*}
U(t, g)=-\nabla^{SG} W(t, g)
\end{eqnarray*}
be the optimal Markov control, then $U(t, g)$ satisfies the Burgers equation $(\ref{BurgersG2})$ on $SG={\rm SDiff}(M)$. 

Finally, let   
$$u(t, x)=(de_x) U(t, e).$$
Then there exists a function $p\in C^1([0, T]\times M, \mathbb{R})$ such that $(u, p)$ 
satisfies the incompressible Euler equation $(\ref{EME-1})$ and  $(\ref{EME-2})$ on $M$  with an external force. 
\end{theorem}
\begin{proof} The proof is similar to the one of Theorem \ref{Main} by taking $\nu=0$. \end{proof}

\begin{remark}  Following Arnold \cite{Arnold} Chapter 9 Sections 46-48, let $S$ be the action function defined by 
\begin{eqnarray}
S_{q_0, t_0}(q, t)=\int_{\gamma} L(q(t), \dot q(t), t) dt,
\end{eqnarray}
along the extremal $\gamma$ connecting the points $(q_0, t_0)$ and $(q, t)$. Then the differential of $S$ (for a fixed initial point) is equal to
\begin{eqnarray}
dS=pdq-Hdt, \label{HJS0}
\end{eqnarray}
where $p={\partial L\over \partial \dot q}$ and $H=p\dot q-L$.
Moreover,  $S$ satisfies the 
Hamilton-Jacobi equation 

\begin{eqnarray}
{\partial S\over \partial t}+H\left({\partial S\over \partial q}, q, t\right)=0. \label{HJS1}
\end{eqnarray}
Indeed, it  is sufficient to notice that
\begin{eqnarray}
{\partial S\over \partial t}=-H(p, q, t), \ \ p=
{\partial S\over \partial q}. \label{HJS2}
\end{eqnarray}
As pointed out by Arnold \cite{Arnold}, ``The relation just established between trajectories of mechanical systems (``rays") and partial differential equations (``wave fronts") can be used in two directions. First, solutions of Equation $(\ref{HJS1})$ can be used for integrating the ordinary differential equations of dynamics. Jacobi's method of integrating Hamilton's canonical equations, presented in the next section\footnote{See Section 47 in \cite{Arnold}.}, consists of just this.
Second, the relation of the ray and wave points of view allows one to reduce integration of the partial differential equations $(\ref{HJS1})$ to integration of a hamiltonian system of ordinary differential equations.''
See \cite{Arnold} p. 256 for more discussions on generating functions $S$ and the Hamilton-Jacobi equation. 

In our situation of the infinite dimensional group $SG={\rm SDiff}(M)$ of volume preserving diffeomorphisms , the action function $S$ is nothing but the value function $W$ that we introduce for the deterministic optimal control problem on $SG={\rm SDiff}(M)$. In view of this, Theorem \ref{Main-Euler} can be regarded as the analogous of the well-known theorem  about the generating function $S$ in Hamiltonian mechanics for the Hamiltonian function $H(\dot q, q, t)={1\over 2}|\dot q|^2+V(q)$ on $T(SG)=T{\rm SDiff}(M)$.
\end{remark}

\medskip
\section{Conclusion and some problems for further work}

\medskip

In conclusion, the main result of this paper can be formulated as follows.

\begin{theorem} \label{Main2} Let $M$ be a compact Riemannian manifold. Let $W$ be a $C^{1, 3}$-smooth solution to the Hamilton-Jacobi-Bellman equation $(\ref{HJB-W})$ on $SG={\rm SDiff}(M)$. Then 
$$u(t, x)=-(de_x)\nabla^{SG} W(t, e)$$
is a $C^{1, 2}$-smooth solution to the backward incompressible Navier-Stokes equations $(\ref{NSM-u1})$, $(\ref{NSM-u2})$. Conversely, let $u$ be a smooth solution to the backward incompressible Navier-Stokes equations $(\ref{NSM-u1})$, $(\ref{NSM-u2})$, and let $\{g_t, t\in [0, T]\}$ be the unique solution to SDE $(\ref{SDEG})$ with $g(0)=e$ on $SG={\rm SDiff}(M)$. Then the value function  $W$ for the  optimal  stochastic control problem $(\ref{W-Gvaluefunction})$ is a solution to the Hamilton-Jacobi-Bellman equation $(\ref{HJB-W})$ on $SG={\rm SDiff}(M)$.
\end{theorem}
\begin{proof} This is indeed the  restatement of our main results in Section 2 and Section 4. \end{proof}

\medskip

To end this paper, we raise some problems for further research work.

\begin{remark} Our main result indicates interesting relationships among the  incompressible Navier-Stokes equations on a compact Riemannian manifold, the Hamilton-Jacobi-Bellman equation and viscous Burgers equation on the group of  volume preserving  diffeomorphisms.
By \cite{EM70, Taylor}, if the initial value $u_0\in C^\infty(M)$, then there exists some $T>0$ such that the incompressible Navier-Stokes equations $(\ref{NSM-u1})$, $(\ref{NSM-u2})$ has a local smooth solution $u\in C^\infty([0, T]\times M)$. This yields that the vector fields in SDE $(\ref{SDEG})$ are smooth on $SG$ and hence SDE $(\ref{SDEG})$ admits a unique strong solution on $[0, T]\times SG$. It is interesting to ask the question whether we can prove the regularity of the solution to the Hamilton-Jacobi-Bellman equation  $(\ref{HJB-W})$ on the group of  volume preserving  diffeomorphisms.

There exist many references on the theory of Hamilton-Jacobi-Bellman equations in finite dimensional spaces and its applications to stochastic optimal control of finite dimensional processes,  such as \cite{Fleming-Rishel, Fleming-Soner, Nisio} etc. However, only few results have been obtained in infinite dimensional spaces.
 For the references of deterministic optimal control  in infinite dimensional case,
 we mention \cite{B-D,D-Z, Nualart} and references therein.
 In \cite{FGS}, the authors give an overview of the theory of Hamilton-Jacobi-Bellman equations  in infinite dimensional Hilbert spaces and its applications to stochastic optimal control of infinite dimensional processes.  In particular, the viscosity solution and mild solution of the Hamilton-Jacobi-Bellman equations  in infinite dimensional Hilbert spaces are studied there.
We would like to point out that our case is not included in \cite{FGS}.  We give the ideas for our future study on the regularity of Hamilton-Jacobi-Bellman equation
as follows.

By the Cole-Hopf transformation 
$$\Phi(t, g)=\exp{\left[-{ W(t, g )\over {2\nu}}\right]}$$
and its inverse transform, 
it is well-known that, $W$ is a solution to the Hamilton-Jacobi-Bellman equation  $(\ref{HJB-W})$ on $SG={\rm SDiff}(M)$ if and only if $\Phi$ is a solution to the heat equation 
 \begin{eqnarray}\label{HeatVG}
 \partial _t \Phi+\nu \Delta^{SG} \Phi+\frac{1}{2\nu}V\Phi=0.
 \end{eqnarray} 
 It remains a challenge problem whether we can prove the 
existence and uniqueness of local or global smooth solution of the Hamilton-Jacobi-Bellman equation  $(\ref{HJB-W})$ and the heat equation $(\ref{HeatVG})$ on $SG={\rm SDiff}(M)$. Instead of smooth solution,  it is also interesting to study viscosity solution or solution in suitable Sobolev space in the sense of the Malliavin calculus \cite{Ma97} over $SG={\rm SDiff}(M)$. This will be our future  research project. See related works \cite{EM, HM, MP} and references therein.

\end{remark}

Finally, let us mention that, after the first version of this paper was submitted, we have finished a more recent  paper \cite{Li-Liu} entitled ``On the Lagrange multiplier method to the Euler and Navier-Stokes equations on Riemannian manifolds'', in which we give a new derivation of the incompressible 
Navier-Stokes equation on  a compact Riemannian manifold $M$ via the Bellman dynamic programming principle on the infinite dimensional group $G={\rm Diff}(M)$ of diffeomorphisms,  where the pressure is explicitly given by a variant of the Lagrange multiplier for the incompressible condition ${\rm{div}}~u=0$.

\medskip

\noindent{\bf Acknowledgement}. The authors would like to express their gratitudes to Professors A. B. Cruzeiro, F. Flandoli and T. Funaki  for their interests and stimulating  discussions in the beginning step of this work. We also thank Dr. S. Li for helpful discussions in the preparation of this work. Finally, we would like to express our sincere gratitude to the associated Editor and an anonymous referee for his/her careful reading and his/her very nice comments which lead us to revise and improve our paper. 

\providecommand{\bysame}{\leavevmode\hbox to3em{\hrulefill}\thinspace}

\begin{flushleft}

Xiang-Dong Li, 1. State Key Laboratory of Mathematical Sciences, Academy of Mathematics and Systems Science, Chinese Academy of Sciences, No. 55, Zhongguancun East Road, Beijing, 100190,  China;\\
2. School of Mathematical Sciences, University of Chinese Academy of Sciences, Beijing, 100049, China\\
E-mail: xdli@amt.ac.cn
\medskip

Guoping Liu, School of Mathematics and Statistics, Huazhong University of Science and Technology,  No.1037, Luoyu Road, Wuhan, 430074, China

E-mail: liuguoping@hust.edu.cn

\end{flushleft}

\end{CJK*}

\end{document}